\documentclass[openany, 12pt]{amsart} 

\usepackage{geometry} 
\usepackage{etoolbox}

\patchcmd{\tocchapter}{#1}{\MakeUppercase{#1}}{}{}

\makeatletter
\patchcmd{\@tocline}
  {\hfil}
  {\leaders\hbox{\,.\,}\hfil}
  {}{}

\newcommand{\R}{\mathbb{R}}  
\newcommand{\Z}{\mathbb{Z}}
\newcommand{\N}{\mathbb{N}}
\newcommand{\Q}{\mathbb{Q}}
\newcommand{\C}{\mathbb{C}}

\newcommand{\F}{\mathbb{F}}

\usepackage{mathrsfs}
\usepackage{hyperref}
\hypersetup{
    colorlinks=false,
    linkcolor=black,
    filecolor=blue,      
    urlcolor=blue,
    citecolor=black
    }
\DeclareFontFamily{U}{mathx}{}
\DeclareFontShape{U}{mathx}{m}{n}{<-> mathx10}{}
\DeclareSymbolFont{mathx}{U}{mathx}{m}{n}
\DeclareMathAccent{\widehat}{0}{mathx}{"70}
\DeclareMathAccent{\widecheck}{0}{mathx}{"71}
\usepackage[table, dvipsnames]{xcolor}
\usepackage{booktabs} 
\usepackage{array} 
\usepackage{paralist} 
\usepackage{verbatim} 
\usepackage{graphicx} 
\usepackage{tikz-cd}
\usepackage{siunitx}
\usepackage{amsfonts}
\usepackage{amssymb}
\usepackage{enumitem}
\usepackage{moreenum}
\usepackage{mathtools}
\usepackage{upgreek}
\usepackage{physics}
\usepackage{amssymb,amscd,amsmath,amsthm,url}
\usepackage{thmtools}
\usepackage{caption}
\usepackage{subcaption}
\usepackage{esint}
\usepackage{sseq}
\usepackage[nameinlink]{cleveref}
\usepackage{needspace}
\setlist[enumerate]{label=\textbf{\arabic*.}, leftmargin = 0.8cm}

\numberwithin{equation}{section}

\usepackage{cancel}

\newcommand{\pd}{\partial}
\newcommand{\can}{\mathrm{can}}
\newcommand{\pre}{\mathrm{pre}}
\newcommand{\w}{\omega}
\newcommand{\std}{\mathrm{std}}

\newcommand\ep{\varepsilon}

\newcommand{\GL}{\textup{GL}}

\newcommand{\U}{\textup{U}}

\newcommand{\Sp}{\textup{Sp}}

\newcommand\Hom{\textup{Hom}}

\newcommand{\ind}{\textup{ind}}
\newcommand{\im}{\operatorname{im}}
\newcommand{\coker}{\operatorname{coker}}
\newcommand{\vol}{\mathrm{vol}}
\setlist[itemize]{leftmargin=*}

\renewcommand{\bar}{\overline}
\renewcommand{\tilde}{\widetilde}
\renewcommand{\hat}{\widehat}
\usepackage{stmaryrd}

\newcommand{\ECH}{\mathrm{ECH}}
\newcommand{\Alt}{\mathrm{Alt}}
\newcommand{\tame}{\mathrm{tame}}

\numberwithin{equation}{section}

\newtheorem{theorem}{Theorem}[section]
\newtheorem{proposition}[theorem]{Proposition}
\newtheorem{corollary}[theorem]{Corollary}
\newtheorem{lemma}[theorem]{Lemma}
\newtheorem{conjecture}[theorem]{Conjecture}
\newtheorem{maintheorem}{Theorem}

\theoremstyle{definition}
\newtheorem{definition}[theorem]{Definition}
\newtheorem{remark}[theorem]{Remark}
\newtheorem{example}[theorem]{Example}
\newtheorem{question}[theorem]{Question}

\allowdisplaybreaks

\usepackage{adjustbox}

\usepackage{csquotes}

\definecolor{mypink2}{RGB}{255, 186, 239}
\definecolor{myyellow}{RGB}{255, 227, 89}
\definecolor{mygreen}{RGB}{178, 252, 141}

\makeatletter
\def\l@subsection{\@tocline{2}{0pt}{2.5pc}{5pc}{}}
\makeatother

\begin{document}
\title[The ECH and Alt Capacities of Closed Symplectic 4-Manifolds]{The ECH and Alternative ECH Capacities\\ of Closed Symplectic 4-Manifolds}

\author{Gabriel Beiner}
\address{Gabriel Beiner, University of California Berkeley}
\email{gabriel.beiner@berkeley.edu}
\begin{abstract}
	We show that the ECH capacities of every closed symplectic 4-manifold are infinite. We also give the first examples of symplectic 4-manifolds whose alternative ECH capacities are infinite. We verify the ECH Weyl law holds with bounded subleading asymptotics for the alternative ECH capacities of any closed symplectic 4-manifold with $b_2^+=1$ or any smooth domain therewithin. 
\end{abstract}

\maketitle

{ \setcounter{tocdepth}{2}
\tableofcontents  }

\section{Introduction}
\subsection{Results on the ECH capacities}
If $(X,\w)$ is a symplectic 4-manifold, one can associate to it a collection of numerical invariants called the \emph{ECH capacities}:
\[0<c_1^{\ECH}(X,\w) \leq c_2^{\ECH}(X,\w) \leq \cdots \leq c_k^\ECH(X,\w) \leq \cdots \leq +\infty.\]
These were introduced by Hutchings \cite{HutQ} using his theory of embedded contact homology (ECH). The ECH capacities are foremost a tool to study embedding problems, since a symplectic embedding $(X,\w) \hookrightarrow (X',\w')$ implies $c_k^\ECH(X,\w) \leq c_k^\ECH(X',\w')$ for all $k$. 

The ECH capacities are quite powerful, for example they recover the optimal obstructions to embedding a union of concave toric domains into a convex toric domain \cite{CG}; in particular they completely determine symplectic embeddings between ellipsoids, which is already a surprisingly subtle problem. The ECH capacities of a Liouville domain, when finite, also satisfy a Weyl law \cite{Weyl}, which shows they asymptotically recover the volume of $(X,\w)$. The capacities are defined in terms of the Reeb dynamics of contact 3-manifolds and as such are also closely related to the study of 3-dimensional Reeb and Hamiltonian flows (see e.g. \cite{CGH,Iri,IriEqui,Bei}).

Computing the ECH capacities can be a formidable challenge, and is really only feasible when $(X,\w)$ is a Liouville domain with a highly symmetric Reeb flow on its boundary. The capacities are now known for convex and concave toric domains \cite{HutQ, CCGFHR}, and the unit cotangent disk bundles of $S^2$ \cite{FR, FRV}, $\R P^2$ \cite{FR}, $T^2$ \cite{HutQ}, and the Klein bottle \cite{MR} with respect to nice metrics. 

For general symplectic 4-manifolds, we can ask the more basic question of whether the ECH capacities are even finite. Recently, it was shown that the capacities are infinite for many examples, including cotangent disk bundles of higher genus surfaces and certain closed symplectic 4-manifolds like the 4-torus, while they are finite for cotangent disk bundles of non-orientable surfaces \cite{Bei}. It has remained open to determine the ECH capacities of even very simple closed symplectic 4-manifolds, like $\C P^2$ and $S^2 \times S^2$ (see \cite[p.\!\! 2]{HutAlt}). In part, this is because the capacities of closed manifolds are defined indirectly as a supremum over the capacities of embedded Liouville domains.  We begin by showing the ECH capacities are infinite for every closed symplectic 4-manifold.

\begin{maintheorem}\label[theorem]{thm:closed_inf}
Let $(X,\w)$ be a closed symplectic 4-manifold. The ECH capacities of $(X,\w)$ are all infinite.
\end{maintheorem}

This result has been independently proven by Chen \cite{CheU} under the additional assumption that the symplectic form is rational.\footnote{In the first version of this paper, which appeared around the same time as Chen's work, we also assumed the symplectic form was rational. Subsequently, we figured out additional arguments to remove that assumption.} Recall we say a symplectic form $\w$ is \emph{rational} if $\lambda \w$ lives in the integral cohomology lattice for some $\lambda>0$.

\begin{remark} \Cref{thm:closed_inf} should not be thought to imply that ECH is totally oblivious to symplectic embedding problems into closed 4-manifolds. Indeed, Hutchings \cite{HutTQFT} has introduced \emph{completed ECH capacities}, which see more of the ECH cobordism map than the ECH capacities. These completed capacities will be finite for some closed symplectic 4-manifolds like $\C P^2$ and $S^2\times S^ 2$ and provide  obstructions to symplectic embeddings into them.	 See also the results of \cite{CW} which imply the ECH capacities obstruct embeddings of star-shaped domains into certain algebraic surfaces.
\end{remark}

One can prove a stronger result under topological assumptions on $(X,\w)$. \begin{maintheorem}\label[theorem]{thm:posinf}
	Suppose $(X,\w)$ is a closed rational symplectic 4-manifold which satisfies $c_1(TX,\w)\cdot [\w]\leq 0$. Then $(X,\w)$ admits an embedded Liouville subdomain whose ECH capacities are infinite.
\end{maintheorem}
The condition $c_1(TX,\w)\cdot [\w]\leq 0$ applies to the large majority of symplectic 4-manifolds. If $(X,\w)$ instead satisfies $c_1(TX,\w) \cdot [\w]>0$, then $b_2^+(X)=1$, and $(X,\w)$ is symplectically uniruled and K\"ahler with Kodaira dimension $-\infty$ (see \cite{MS96} and the appendix of \cite{LN}).

We expect \Cref{thm:posinf} to hold without the rationality assumption. As a converse to \Cref{thm:posinf}, we conjecture the following, some evidence for which is obtained by Mark--Tosun \cite{MT}; we discuss this further after the proof of \Cref{thm:posinf}.
\begin{conjecture}\label{conj:posinf}
\Cref{thm:posinf} is sharp in the following sense. Suppose $(X,\Omega)$ is a closed symplectic 4-manifold with $c_1(TX,\Omega)\cdot [\Omega]>0$. If $(W,\w)$ is a four-dimensional Liouville domain with a symplectic embedding $(W,\w)\hookrightarrow (X,\Omega)$, then $c_k^{\ECH}(W,\w)<\infty$ for all $k$. 	
\end{conjecture}
\Cref{conj:posinf} holds when $(W,\w)$ is Weinstein by combining work of Mark--Tosun and Chen \cite{MT,CheU}.

\subsection{Results on the alternative ECH capacities}

Next, we study the \emph{alternative ECH capacities} of a symplectic 4-manifold $(X,\w)$:
\[0<c_1^{\Alt}(X,\w) \leq c_2^{\Alt}(X,\w) \leq \cdots \leq c_k^\Alt(X,\w) \leq \cdots \leq +\infty,\]
which were recently introduced by Hutchings \cite{HutAlt}. These capacities are defined by abstract max-min problems over the space of point-constrained holomorphic curves. They mimic the properties of the ECH capacities with a considerably simpler definition. For example, one can use these capacities to recover the optimal obstructions to symplectic ellipsoid embeddings, and to reprove the existence of at least two orbits and Irie's $C^\infty$ closing lemma for Reeb flows on the boundary of star-shaped domains in $\R^4$ and on the cotangent circle bundles of $S^2, T^2$, and $\R P^2$ (following the arguments of \cite{HutAlt,Iri, CGH, Hutqc}). The proofs require no Seiberg--Witten theory or the usual difficulties of ECH and SFT; they can be deduced just from Hutchings' ECH index inequality \cite{Hutindineq},  Gromov--Taubes compactness \cite{Tau}, and basic pseudoholomorphic curve analysis known since Gromov's first work (see the computations of \cite[\S B.2]{EdtPFH}). As such, understanding and computing the alternative capacities is highly desirable.

The alternative capacities are bounded above by the ECH capacities \cite{HutAlt}. In many cases the two agree, although some cases are known where they differ \cite[Rmk.\!\! 14]{HutAlt} (many more such cases follow from comparing \Cref{thm:closed_inf} with \Cref{thm:alt_fin}). In general, lower bounds on the alternative capacities are difficult to obtain (beyond those directly implied by the properties in \cite[Thm.\!\! 6]{HutAlt}), since they require ruling out the existence of holomorphic curves, even those not detected by enumerative and Floer-theoretic invariants. In \cite[Rem.\!\ 3.11]{Hut26}, the question was raised of whether the alternative ECH capacities are always finite. Here we show that the alternative capacities are sometimes infinite, at least for closed symplectic 4-manifolds.

To understand the statement, we first recall a well-known construction of Zehnder \cite{Zeh}. We study a constant coefficient symplectic form on $T^4= \R^4/\Z^4$ with coordinates $x_1,\ldots, x_4$:
\begin{equation}\label{eq:Zehtor}
	\w = \frac{1}{2}\sum_{i,j} A_{ij} \dd x_i\wedge \dd x_j,
\end{equation}
where $A$ is any invertible skew-symmetric $4\times 4$ matrix. If we restrict to the hypersurface $Y=\{x_4=c\}$, the symplectic form restricts to
\begin{equation}\label{eq:Zehhyp}
\w|_Y = A_{12} \dd{x_1}\wedge \dd{x_2} + A_{13} \dd{x_1}\wedge \dd{x_3} + A_{23} \dd{x_2}\wedge \dd{x_3}.	
\end{equation}
The 2-form $\w|_Y$ defines a \emph{(stable) Hamiltonian structure} on the 3-torus $Y$. Such a structure gives rise to canonical dynamics, namely a \emph{characteristic Hamiltonian flow} tangent to $\ker(\w|_Y)$.  If the components of the vector 
\begin{equation}\label{eq:Zehhypcond}
	(A_{12}, A_{13}, A_{23})\in \R^3 \quad \text{are linearly independent over $\Q$,}
\end{equation} then the characteristic flow is an irrational rotation of the 3-torus. This flow is minimal: every orbit is dense, and it has no periodic orbits. If one moves to a nearby hypersurface $\{x_4=c\pm \varepsilon\}$, the characteristic flow will be the same, and hence still have no periodic orbits. Thus, a symplectic form \eqref{eq:Zehtor} on $T^4$ satisfying \eqref{eq:Zehhypcond}, often called a \emph{Zehnder torus}, will fail the \emph{nearby existence property}, and its Hofer--Zehnder capacity will be infinite \cite[Ch.\!\! 4]{HZ}.

\begin{definition}
A  hypersurface $Y$ in a symplectic 4-manifold $(X,\w)$ is called a \emph{Zehnder torus hypersurface} if $Y$ is a 3-torus and there are coordinates $x_1,x_2,x_3: Y\to S^1$ on $Y$ so that $\w|_Y$ has the form \eqref{eq:Zehhyp} with coefficients satisfying \eqref{eq:Zehhypcond}.  
\end{definition}

Almost every constant coefficient symplectic form on the $4$-torus admits Zehnder torus hypersurfaces. Following a construction of Usher \cite{Ush}, one can also often find these hypersurfaces by starting with a symplectic manifold obtained by a Gompf fibre sum along symplectic tori with trivial normal bundle and perturbing the symplectic form near the sum region. In particular, there will be symplectic forms on any simply connected elliptic surface with $b_2^+>1$ that have Zehnder torus hypersurfaces.

\begin{maintheorem}\label[theorem]{thm:alt_inf}
Let $(X,\w)$ be a closed symplectic 4-manifold containing a Zehnder torus hypersurface. Then the alternative ECH capacities of $(X,\w)$ are infinite.
\end{maintheorem}

After the proof of \Cref{thm:alt_inf}, we introduce what we call the \emph{tame elementary capacities}
\[0<c_1^\tame(X,\w)\leq c_2^\tame(X,\w)\leq \cdots\leq c_k^\tame(X,\w) \leq\cdots \leq  +\infty\] and we show they are infinite for closed symplectic 4-manifolds in quite a bit of generality (see \Cref{prop:tame}). 

To conclude, we contrast \Cref{thm:closed_inf,thm:alt_inf} by observing that the alternative ECH capacities are finite for many closed examples. This is likely known to experts in some form, although we are not aware of anywhere that the Weyl law or its subleading asymptotics are written down.

\begin{maintheorem}\label[theorem]{thm:alt_fin}
	Let $(X,\w)$ be a closed symplectic 4-manifold with $b_2^+(X)=1$. Then the alternative ECH capacities are finite and obey a Weyl law with bounded error term:
	\begin{equation}\label{eq:Weyl}
		c_k^{\Alt}(X,\w) = \sqrt{4\mathrm{vol}(X,\w)k} +O(1) \quad (k\to \infty).
	\end{equation}
\end{maintheorem}
After the proof of \Cref{thm:alt_fin}, we analyse further the error term 
\[e_k^\Alt(X,\w) = c_k^\Alt(X,\w) -\sqrt{4\vol(X,\w)k}\]
and conjecture a general expression for its asymptotic behaviour (see \Cref{conj_ek}). The analogous expression for the ECH capacities has been studied by Hutchings \cite{HutRuelle} in the case of star-shaped domains. Together with Edtmair's results on packing stability \cite{EdtPack}, \Cref{thm:alt_fin} also implies the following.
\begin{corollary}\label{cor:alt_fin}
	Let $(X,\w)$ be a compact symplectic 4-manifold with smooth boundary which symplectically embeds in a closed symplectic 4-manifold with $b_2^+=1$. Then the alternative ECH capacities of $(X,\w)$ are finite and satisfy the Weyl law \eqref{eq:Weyl}.
\end{corollary}
Note that the corresponding statement of \Cref{cor:alt_fin} for ECH capacities is false. For example, if $(X,\w)$ is any closed rational symplectic 4-manifold with 
\[b_2^+(X)=1 \qand c_1(TX,\w)\cdot [\w]\leq 0\] 
(e.g.\! see \Cref{ex:bad_uniruled}), then \Cref{thm:posinf} implies there are embedded Liouville subdomains whose ECH capacities are infinite.

Even when the ECH capacities are finite, it might not be known that the Weyl law \eqref{eq:Weyl} holds. For example, the ECH Weyl law is not known for a blown-up ball $(X,\w) = B^4(1)\# \overline{\C P^2}(\lambda)$; one has at best an approximate Weyl law:
\[ 1-\lambda^2 \leq \lim_{k\to\infty} \frac{c_k^{\ECH}(X,\w)^2}{2k} \leq  1.\]
By embedding $(X,\w)$ into $\C P^2(1)\# \overline{\C P^2}(\lambda)$, \Cref{cor:alt_fin} shows the left inequality is an equality for $c_k^\Alt(X,\w)$. 
    
And even when the leading term of the ECH Weyl law is known, the subleading asymptotics may not be. For example, in \cite[Prop.\!\! 4.10]{Bei}, it is shown that a cotangent disk bundle $(D^\ast N,\w_\can)$ of any closed non-orientable surface $N$ has finite ECH capacities which obey a Weyl law. But it is only known that the error term in the Weyl law is $O(1)$ for non-orientable surfaces with Euler characteristic congruent to 0 or 1  mod 4 \cite[Rmk.\!\! 4.12]{Bei}. As in \cite{DHL}, one can embed all cotangent disk bundles over non-orientable surfaces into an iterated blow-up of $\C P^2$ and hence the capacities $c_k^\Alt(D^\ast N,\w_\can)$ will satisfy the precise form of the Weyl law with error term in \eqref{eq:Weyl} for any non-orientable surface $N$.

\addtocontents{toc}{\protect\setcounter{tocdepth}{1}}

\subsection*{Outline} 
In \Cref{sec:defs}, we briefly recall the definitions and properties of the ECH capacities and their elementary alternatives. In \Cref{sec:proofs}, we prove the results stated in the introduction, with a few relevant diversions along the way.

\subsection*{Acknowledgements}
I'm grateful to my advisor Michael Hutchings for his insights and guidance. I'd also like to thank Rohil Prasad and Dan Cristofaro-Gardiner for helpful discussions which inspired and improved this work. I would like to acknowledge the support of the Natural Sciences and Engineering Research Council of Canada (NSERC) through PGS D-587425-2024. Je voudrais remercier le Conseil de recherches en sciences naturelles et en g\'{e}nie du Canada (CRSNG) de son soutien.  

\addtocontents{toc}{\protect\setcounter{tocdepth}{2}}

\section{A recollection on capacities}\label{sec:defs}
In this section, we remind the reader briefly of the definition of the ECH capacities and their elementary alternatives. For a more thorough introduction, see \cite{HutQ,Hutnotes} and \cite{HutAlt,Hut26} respectively.

\subsection{Some definitions}
We begin by recalling some standard definitions and terminology which will be needed to define the capacities and/or will appear in the proofs of our theorems.

\subsubsection{Symplectic-type structures}

\begin{definition}
	Let $Y$ be a closed oriented 3-manifold. Let $\w$ be a closed 2-form on $Y$ and $\lambda$ be a 1-form on $Y$.
	\begin{enumerate}
		\item We say $\w$ is a \emph{Hamiltonian structure} on $Y$ if $\ker(\w)$ is a one-dimensional smooth subdistribution of $TY$.
		\item The pair $(\lambda,\w)$ is a \emph{framed Hamiltonian structure} (FHS) if $\lambda\wedge \w>0$. 
		\item A framed Hamiltonian structure $(\lambda,\w)$ is called a \emph{stable Hamiltonian structure} (SHS) if $\ker(\w) \subset \ker(\dd\lambda)$. Equivalently, $\dd\lambda = f\w$ for some $f:Y\to\R$.
	\end{enumerate}
	Associated to an FHS $(\lambda,\w)$ is a canonical \emph{Hamiltonian vector field} $R$ uniquely specified by the conditions $i_R \w=0$ and $\lambda(R)=1$.
\end{definition}

\begin{example}
	The following are the core examples of FHSs.
	\begin{enumerate}
		\item A 1-form $\lambda$ is contact if and only if $(\lambda, \dd\lambda)$ defines an SHS. In this case, the Hamiltonian vector field is called the \emph{Reeb vector field} of $\lambda$.
		\item If $\phi:(\Sigma,\w)\to (\Sigma,\w)$ is an area-preserving surface diffeomorphism, then the mapping torus 
		\[M_\phi = [0,1]_t\times \Sigma/(0, \phi(x))\sim (1,x)\] 
		inherits an SHS $(\dd{t}, \w)$. The Hamiltonian vector field is $\pd_t$, which generates the suspension flow of $\phi$. Every FHS with a closed framing one-form can be described in this way.
		\item If $X$ is a non-singular volume-preserving vector field on an oriented 3-manifold $Y$ with volume form $\Omega$, then $i_X \Omega$ is a Hamiltonian structure. It is always possible to find some choice of framing. In fact, the Hamiltonian vector field of an FHS ($\lambda,\w$) always preserves the volume form $\lambda \wedge \w$, so really one can view FHSs as a symplectic formalism for the study of 3D volume-preserving flows.
	\end{enumerate}
\end{example}

Let $(X,\Omega)$ be a symplectic 4-manifold and $Y\subset X$ a closed hypersurface. The restriction $\w=\Omega|_Y$ is a Hamiltonian structure. It will admit a framing if and only the hypersurface $Y$ is co-oriented, in which case the space of framings is an affine space modelled on the space of positive functions on $Y$. Given a choice of framing $(\lambda, \w)$, a Moser argument for coisotropic submanifolds implies a sufficiently small tubular neighbourhood $i: U \hookrightarrow X$ of $Y$ can be identified with $(-\varepsilon,\varepsilon)_t\times Y$ so that $\{0\}\times Y= i^{-1}(Y)$ and so that
\[i^\ast \Omega = \w + \dd(t\lambda).\]
The nearby hypersurfaces $\{t\}\times Y$ inherit an FHS $(\lambda, \w+t\dd\lambda)$. The Hamiltonian vector fields on $\{t\}\times Y$ will agree for all $t\in (-\ep,\ep)$ if and only if $(\lambda,\w)$ is stable, whence the name. 

Conversely, given an FHS $(\lambda,\w)$, one can always form an associated symplectic thickening $((-\ep,\ep)_t\times Y, \w+\dd(t\lambda))$  for $\ep>0$ sufficiently small, which we call a \emph{symplectization}. In the contact case, one usually changes coordinates to define a global symplectization $(\R_t\times Y, \dd(e^t\lambda))$.

\begin{definition}
	Given a contact 3-manifold $(Y,\lambda)$, a \emph{Reeb orbit} $\gamma: \R/T\Z \to Y$ is a periodic orbit of the Reeb vector field modulo reparametrization of the domain. The orbit is called \emph{simple} if it is embedded. The Poincar\'e first return map along a Reeb orbit $\gamma$ defines a symplectic linear map $P_\gamma : \xi_{\gamma(0)}\to \xi_{\gamma(0)}$ on the contact distribution $\xi=\ker\lambda$. We say $\gamma$ is \emph{non-degenerate} if 1 is not an eigenvalue of $P_\gamma$. A non-degenerate orbit is either \emph{elliptic}, if $P_\gamma$ has complex eigenvalues, or \emph{positive/negative hyperbolic}, if it has real and positive/negative eigenvalues. We say $\lambda$ is \emph{non-degenerate} if all of its Reeb orbits are non-degenerate; a $C^\infty$-generic contact form is non-degenerate.
	
	For $(Y,\lambda)$ contact, an \emph{orbit set} is a finite set of pairs $\{(\gamma_i,m_i)\}$ where $\gamma_i$ are distinct simple Reeb orbits and $m_i$ are positive integer multiplicities. For $(Y,\lambda)$ non-degenerate, an orbit set is \emph{admissible} provided $m_i=1$ whenever $\gamma_i$ is hyperbolic. 
	
	The \emph{action} $\mathcal{A}(\gamma)$ of a Reeb orbit $\gamma$ is its period $\int_\gamma \lambda$. The action of an orbit set is
	\[\mathcal{A}(\{(\gamma_i,m_i)\}) = \sum_i m_i \int_{\gamma_i}\lambda.\] 
\end{definition}

\begin{definition}
	A compact symplectic manifold $(X,\w)$ with boundary $Y$ is called a \emph{weak Liouville domain} if $\w$ is exact and the boundary has a contact form $\lambda$ with $\dd\lambda=\w|_Y$. Additionally, $(X,\w)$ is a \emph{Liouville domain} if $\lambda$ extends to a global primitive for $\w$. A (weak) Liouville domain is \emph{non-degenerate} if its boundary is a non-degenerate contact manifold. 
	
	The \emph{Liouville vector field} of a Liouville domain $(X,\w=\dd\lambda)$ is the unique vector field $V$ such that $i_V \w = \lambda$; note this implies $\mathcal{L}_V\w=\w$, so that the symplectic form exponentially expands along the flow of $V$. The \emph{Liouville skeleton} $\Delta$ of a Liouville domain is the union of $\alpha$-limit sets of the flow of $V$. A Liouville domain $(X,\w)$ is \emph{Weinstein} if the Liouville vector field is gradient-like for a Morse function $\varphi: X\to \R$ which is maximized on $\pd X$.

Given a weak Liouville domain $(X,\w)$ with boundary $(Y,\lambda)$, we can form its \emph{symplectic completion} $(\widehat{X},\widehat{\w})$ by attaching a cylindrical symplectization end to its boundary:
\[(\widehat{X},\widehat{\w}) = (X,\w) \bigcup_{\pd X = \{0\}\times Y} ([0,\infty)_t \times Y, \dd(e^t\lambda)).\]
\end{definition}

\subsubsection{Almost-complex structures}

\begin{definition}
	Let $J$ be an almost-complex structure on a symplectic manifold $(X,\w)$.	\begin{enumerate}
		\item We say $J$ is \emph{$\w$-tame} if the bilinear form $\w(\cdot , J \cdot)$ is positive definite.
		\item We say $J$ is \emph{$\w$-compatible} if the bilinear form $\w(\cdot ,J\cdot)$ defines a Riemannian metric, i.e. it is both positive definite and symmetric.
	\end{enumerate}
\end{definition}

\begin{definition}
	An almost-complex structure $J$ on a symplectization $(-\ep,\ep)_t\times Y$ of an SHS $(\lambda,\w)$ is called \emph{$(\lambda,\w)$-compatible} provided the following hold:
	\begin{enumerate}
		\item $J$ is invariant under translation by $t$.
		\item $J(\pd_t)=R$, where $R$ is the Hamiltonian vector field.
		\item $J$ preserves $\ker(\lambda) \subset TY$ and satisfies $\w(v,Jv)\geq 0$ for $v \in \ker(\lambda)$.
	\end{enumerate}
	Such a $J$ is necessarily compatible with the symplectic form on the symplectization for $\ep$ sufficiently small and the space of $(\lambda,\w)$-compatible almost-complex structures is in bijection with the space of compatible almost-complex structures on $\ker(\lambda)$ (in particular, it is non-empty).
	
	In the case where our SHS is a contact pair $(\lambda, \dd\lambda)$ and we take the global symplectization $(\R\times Y, \dd(e^t\lambda))$, a $(\lambda, \dd\lambda)$-compatible almost-complex structure is called \emph{$\lambda$-compatible} and it is globally compatible with the symplectic form.
\end{definition}

\subsection{The ECH capacities}
Associated to a closed oriented contact 3-manifold $(Y,\xi)$, there is a Floer theory called \emph{embedded contact homology} (ECH) introduced by Hutchings \cite{Hutnotes}. Given a choice of non-degenerate contact form $\lambda$ for $\xi$, one defines a chain complex $ECC(Y,\lambda, J)$ generated over $\F_2$ by admissible orbit sets. For a generic choice of $\lambda$-compatible almost-complex structure $J$ on the symplectization $\R\times Y$, the differential on $ECC(Y,\lambda, J)$ is defined by counting certain embedded $J$-holomorphic curves of any genus in $\R\times Y$ asymptotic to orbit sets. The homology of the chain complex is denoted $ECH(Y,\lambda)$. In fact, ECH depends only on the underlying contact structure, and as such is also denoted $ECH(Y,\xi)$. And up to a grading shift, ECH is a purely topological invariant, coinciding with certain flavours of Seiberg--Witten and Heegaard Floer theories. We summarize a few key additional structures of ECH that are relevant to the construction of the ECH capacities.

\begin{enumerate}
	\item There is a grading by $H_1(Y;\Z)$:
	\[ECH(Y,\xi) = \bigoplus_{\Gamma \in H_1(Y)} ECH(Y,\xi,\Gamma),\]
	where $ECH(Y,\xi,\Gamma)$ is the homology of the subcomplex of $ECC(Y,\lambda,J)$ generated only by orbit sets whose total homology class is $\Gamma$.
	\item The group $ECH_\ast(Y,\xi,\Gamma)$ has a relative $\Z/d\Z$-grading where $d$ is the divisibility of 
	 \[c_1(\xi)+2\mathrm{PD}(\Gamma) \in H^2(Y;\Z)/\mathrm{torsion}.\]
	 This grading is given by the \emph{ECH index} $I(\alpha,\beta)$ of a pair of orbit sets $\alpha,\beta$. In particular, whenever $c_1(\xi)$ is torsion, $ECH_\ast(Y,\xi,0)$ has a canonical $\Z$-grading, where $I(\gamma)=I(\gamma,\emptyset)$. In addition, one can always reduce to a canonical $\Z/2\Z$-grading on $ECH_\ast(Y,\xi,\Gamma)$, which on the chain level counts the parity of the number of positive hyperbolic orbits in an orbit set.  
	\item There is a special element $c(Y,\xi)\in ECH_0(Y,\lambda,0)$ sensitive to the contact structure called the \emph{contact invariant}. It is the homology class $[\emptyset]$ of the empty set of Reeb orbits. This is always a cycle but it may be zero in homology, e.g. when $(Y,\xi)$ is overtwisted. It is non-vanishing, however, whenever $(Y,\xi)$ is strongly fillable \cite{HutTQFT}.
	\item If $Y$ is connected, there is a \emph{U map}
	\[U: ECH_\ast(Y,\lambda, \Gamma) \to ECH_{\ast-2}(Y,\lambda, \Gamma)\]
	which is induced from a chain map counting embedded $J$-holomorphic curves in $\R\times Y$ that pass through a point constraint.
	\item By Stokes' theorem, the ECH differential will always map orbit sets to orbit sets of lower action, hence, for $L\in\R$, we can make sense of
	\[ECH_\ast^{<L}(Y,\lambda,\Gamma),\]
	which is the homology of the subcomplex $ECC^{<L}_\ast(Y,\lambda,J,\Gamma)$ of $ECC_\ast(Y,\lambda,J,\Gamma)$ generated by all orbit sets $\gamma$ with $\mathcal{A}(\gamma) <L$. The $U$ map also respects this $\R$-filtration. While ECH is a topological invariant, the action filtration is highly sensitive to $\lambda$.
\end{enumerate}

\begin{definition}
The \emph{ECH spectrum} of a non-degenerate connected contact 3-manifold $(Y,\lambda)$ is the increasing sequence
\[0=c_0(Y,\lambda)<c_1(Y,\lambda)\leq c_2(Y,\lambda)\leq\cdots \leq c_k(Y,\lambda) \leq \cdots \leq +\infty\]
defined by 
\[c_k(Y,\lambda) = \inf\{L>0 : [\emptyset]\in \im\left(U^k: ECH^{<L}(Y,\lambda,0)\to ECH^{<L}(Y,\lambda,0)\right)\}.\]
We extend this definition to a degenerate contact form $\lambda$ as the limit of the ECH spectrum of any sequence of non-degenerate contact forms $C^\infty$-converging to $\lambda$ (and this will be independent of choice). 

If $(Y,\lambda)=(Y_1,\lambda_1)\sqcup\cdots \sqcup (Y_n,\lambda_n)$ is disconnected, we set
\[c_k(Y,\lambda)= \max_{k_1+\cdots+k_n=k} \sum_{i=1}^n c_{k_i}(Y_i,\lambda_i).\]
\end{definition}

\begin{definition}
	The \emph{ECH capacities} of a weak Liouville domain $(X,\w)$ with contact boundary $(Y,\lambda)$ is the sequence 
	\[0=c_0^\ECH(X,\w)<c_1^\ECH(X,\w)\leq c_2^\ECH(X,\w)\leq\cdots \leq c_k^\ECH(X,\w) \leq \cdots \leq +\infty\]
defined by 
\[c_k^\ECH(X,\w) = c_k(Y,\lambda).\]
Given an arbitrary symplectic 4-manifold $(X,\w)$, we define its $k$th ECH capacity to be the supremum of the $k$th ECH capacity of any union of weak Liouville domains which symplectically embeds into $(X,\w)$.
\end{definition}

\begin{proposition}[\cite{HutQ, Weyl}]
	The ECH capacities have the following properties.
	\begin{enumerate}[itemsep=0.3cm]
		\item (Monotonicity) Given a symplectic embedding $(X,\w) \hookrightarrow (X',\w')$,
		\[ c_k^{\mathrm{ECH}}(X,\w) \leq c_k^{\mathrm{ECH}}(X',\w')\quad \text{ for all $k$.} \]
		\item (Conformality) For any $s>0$,
		\[ c_k^{\mathrm{ECH}}(X,s\w) = s c_k^{\mathrm{ECH}}(X,\w) \quad \text{for all $k$.}\]
		\item(Spectrality) For a weak Liouville domain $(X,\w)$ with contact boundary $(Y,\lambda)$, if $c_k^\ECH(X,\w)<\infty$, then there is some orbit set $\gamma$ with $[\gamma]=0\in H_1(Y)$ and $\mathcal{A}(\gamma)=c_k^\ECH(X,\w)$. If $(X,\w)$ is non-degenerate, then this orbit set can be taken to be admissible.
		\item(ECH index) Suppose $(X,\w)$ is a weak Liouville domain with non-degenerate boundary so that $c_1(TX,\w)|_{\pd X}$ is torsion and $c_k^\ECH(X,\w)<\infty$. Then the orbit set $\gamma$ in the spectrality property can be taken to have ECH index $I(\gamma)=2k$.
		\item (Disjoint union)
		\[c_k^{\ECH}\left(\coprod_{i=1}^n (X_i,\w_i)\right) = \max_{k_1+\cdots +k_n=k} \sum_{i=1}^n c_{k_i}^\ECH(X_i,\w_i).\]
		\item (Weyl law) If $(X,\w)$ is a weak Liouville domain whose ECH capacities are all finite, then 
		\[\lim_{k\to \infty} \frac{c_k^{\mathrm{ECH}}(X,\w)^2}{4k} = \mathrm{vol}(X,\w):= \frac{1}{2}\int_X \w\wedge\w.\]
	\end{enumerate}
\end{proposition}

\subsection{The alternative ECH capacities}
We will say a symplectic 4-manifold is \emph{admissible} if it is a disjoint union of closed symplectic manifolds and non-degenerate weak Liouville domains. Given an admissible symplectic 4-manifold $(X,\w)$, we define $(\widehat{X},\widehat{\w})$ to be the symplectic manifold obtained by taking the symplectic completion of any Liouville components.
\begin{definition}
 
For $(X,\w)$ an admissible symplectic 4-manifold, an almost-complex structure $J$ on $(\widehat{X},\widehat{\w})$ is called \emph{admissible} if it is $\w$-compatible, and agrees with the restriction of a $\lambda$-compatible almost-complex structure on the cylindrical ends of any Liouville components with Liouville form $\lambda$. Given an admissible symplectic 4-manifold $(X,\w)$, we let $\mathcal{J}(X,\w)$ denote the collection of admissible almost-complex structures on $(\widehat{X},\widehat{\w})$.\end{definition}

\begin{definition}
	For $(X,\w)$ admissible and $J\in \mathcal{J}(X,\w)$, we consider pseudoholomorphic maps 
	\[u: (\Sigma, j) \to (\widehat{X},J)\]
	where $(\Sigma, j)$ is a closed (possibly disconnected) Riemann surface with finitely many punctures. We require that near each puncture of $\Sigma$, there is a closed neighbourhood on which $u$ asymptotically approaches the cylinder $[0,\infty)\times \gamma$ over some Reeb orbit $\gamma$ on the boundary of $X$.

	Given $(X,\w)$ admissible, $J\in \mathcal{J}(X,\w)$, and a collection of $k$ distinct points $x_1,\ldots, x_k \in X$, we define
	\[\mathcal{M}^J(X,\w; x_1,\ldots, x_k)\]
	to be the set of all $J$-holomorphic curves in $(\hat{X},\hat{\w})$ of the form prescribed above which have the points $x_1,\ldots, x_k$ in their image and which are non-constant on all components of their domain, modulo biholomorphic reparametrization of their domain.
\end{definition}

\begin{definition}
	Given $u \in \mathcal{M}^J(X,\w)$, its \emph{energy} is defined to be
	\[\mathcal{E}(u) = \int_\Sigma u^\ast\w.\]
	In the case where $(X,\w)$ is closed, this is the homological pairing of $[\w]$ and $u_\ast[\Sigma]$. In the case where $(X,\w)$ is a Liouville domain, Stokes' theorem implies this is equal to the sum of the actions of the Reeb orbits which $u$ is positively asymptotic to.
\end{definition}

\begin{definition}
	The \emph{alternative ECH capacities} of an admissible symplectic 4-manifold $(X,\w)$ is the sequence
	\[0 = c_0^\Alt(X,\w) < c_1^\Alt(X,\w) \leq  c_2^\Alt(X,\w) \leq \cdots \leq c_k^\Alt(X,\w)\leq \cdots \leq +\infty\]
	defined by the max-min problems
	\[c_k^\Alt(X,\w) = \sup_{\substack{J\in \mathcal{J}(X,\w)\\ x_1,\ldots , x_k \in X \text{ distinct}}} \inf_{u\in \mathcal{M}^J(X,\w;x_1,\ldots, x_k)} \mathcal{E}(u).\]
	For an arbitrary symplectic 4-manifold $(X,\w)$, we define its $k$th alternative ECH capacity to be the supremum of the $k$th alternative capacity of any admissible symplectic 4-manifold which symplectically embeds into $(X,\w)$. 
\end{definition}

\begin{proposition}[\cite{HutAlt}]\label{prop:altproperties}
	The alternative ECH capacities have the following properties.
	\begin{enumerate}[itemsep=0.3cm]
		\item (Monotonicity) Given a symplectic embedding $(X,\w)\hookrightarrow (X',\w')$,
		\[c_k^\Alt(X,\w) \leq c_k^\Alt(X',\w') \quad \text{for all $k$.}\]
		\item (Conformality) For any $s>0$,
		\[c_k^\Alt(X,s\w) = sc_k^\Alt(X,\w)\quad \text{for all $k$.} \]
	 \item(Spectrality) If $(X,\w)$ is a Liouville domain with contact boundary $(Y,\lambda)$, and if $c_k^\Alt(X,\w)<\infty$, then there exists some orbit set $\gamma$ on $Y$, with $[\gamma]=0\in H_1(X)$ and $\mathcal{A}(\gamma)=c_k^\Alt(X,\w)$. 
	 \item(ECH index) Suppose $(X,\w)$ is a Liouville domain with non-degenerate boundary so that $c_1(TX,\w)|_{\pd X}$ is torsion and $c_k^\Alt(X,\w)<\infty$. Then the orbit set $\gamma$ in the spectrality property can be taken to have ECH index $I(\gamma)\geq 2k$.
		\item (Disjoint union)
		\[c_k^{\Alt}\left(\coprod_{i=1}^n (X_i,\w_i)\right) = \max_{k_1+\cdots +k_n=k} \sum_{i=1}^n c_{k_i}^\Alt(X_i,\w_i).\]
		\item (Sublinearity) 
		\[c_{k+\ell}^\Alt(X,\w) \leq c_k^\Alt(X,\w)+c_\ell^\Alt(X,\w) \quad \text{for all $k,\ell\geq0$.}\]
		\item (ECH bound) If $(X,\w)$ is a Liouville domain,
		\[c_k^\Alt(X,\w) \leq c_k^\ECH(X,\w) \quad \text{for all $k$.}\]
		\item (Seiberg--Witten bound) Let  $(X,\w)$ be a closed symplectic 4-manifold with $b_2^+=1$ and $A\in H_2(X)$. If the Seiberg--Witten invariant $SW_+(X,\mathfrak{s}_\w +A)$ for the positive chamber as determined by $[\w]$ is non-zero, then
		\[c_k^\Alt(X,\w) \leq \langle [\w],A\rangle.\]
	\end{enumerate}
\end{proposition}

\section{Proofs of the main theorems}\label{sec:proofs}

\subsection{(Singular) Donaldson divisors and infinite ECH capacities}

In this section we prove \Cref{thm:closed_inf,thm:posinf}. The key input will be the existence of Donaldson divisors and a Biran decomposition for integral symplectic manifolds,  and Opshtein's extension to singular divisors in non-integral symplectic manifolds.

We begin with a closed rational symplectic 4-manifold $(X,\w)$. Because the ECH capacities are monotone under rescaling, we may rescale our rational symplectic form so that $[\w]$ is in the integral cohomology lattice. In this situation, a famous result of Donaldson \cite{Don} guarantees that for every sufficiently large integer $k$, there exists a closed embedded symplectic submanifold $\Sigma_k \subset X$ so that the homology class $[\Sigma_k] \in H_2(X;\Z)$ is Poincar\'{e} dual to $k[\w]$. Such a submanifold is what we call a \emph{Donaldson divisor}. It is sometimes also called a symplectic hyperplane section, because, in analogy with the Kodaira embedding theorem, the submanifolds are produced as the zero set of nearly-pseudoholomorphic sections of complex line bundles $L^{\otimes k}$, where $c_1(L)=[\w]$. 

As observed by Biran \cite{Bir} and further explicated by Giroux \cite{Gir}, Donaldson divisors give rise to distinguished decompositions of the ambient symplectic manifold. By this we mean that we may decompose $(X,\w)$ as a union:
	\begin{equation}\label{eq:Biran} (X,\w) = (M,\dd\lambda) \bigcup_{\pd M=-\pd\nu(\Sigma_k)} (\nu(\Sigma_k), \w_\std).
	\end{equation}
	Here $(M,\dd\lambda)$ is a Liouville domain\footnote{In fact, as Giroux \cite{Gir} shows, $(M,\dd{\lambda})$ can be taken to be Weinstein for $k$ large enough, but this will not matter for us.} and $(\nu(\Sigma_k),\w_\std)$ is a symplectic disk bundle over $\Sigma_k$, and the two domains are glued along their common contact-type boundary. The \emph{Biran decomposition} \eqref{eq:Biran} is symplectic in the sense that $\w$ agrees with $\dd\lambda$ and $\w_\std$ when restricted to $M$ and $\nu(\Sigma_k)$ respectively. Note it is possible that the Donaldson divisor is disconnected (although for $k$ large enough one can always find a connected Donaldson divisor \cite{Gir}). We will assume for the moment that our Donaldson divisor is connected, and address the disconnected case in \Cref{lem:Dondisc}. When the divisor is connected, we can see that the common boundary $Y=\pd M=-\pd\nu(\Sigma_k)$ is an $S^1$-bundle over $\Sigma_k$ of Euler number $-[\Sigma_k]^2$. It will be equipped with its prequantization contact structure, which we define momentarily.
	
	First note that we may change the relative size of the two pieces of the Biran decomposition by translating $Y$ along its Liouville flow. On the one hand, we may shrink $M$ onto its Liouville skeleton $\Delta$; the complement $X\setminus \Delta$ is a maximal open symplectic disk bundle over $\Sigma_k$ with fibre of area $ k^{-1}$ (this was Biran's original interest towards applications in symplectic embeddings). On the other hand, we may shrink the disk bundle onto the base $\Sigma_k$. Then $X\setminus \Sigma_k$ will be an open Liouville manifold, with an ideal boundary $Y$ giving a natural compactification to a Liouville domain which we call $(M_{\max},\dd\lambda)$. 

Let $(\Sigma, \sigma)$ be a closed oriented surface with $\sigma$ an area form so that $[\sigma]\in H^2(\Sigma)$ lives in the integral cohomology lattice. Let $p: Y\to \Sigma$ be the $S^1$-bundle over $\Sigma$ with Euler class $e = -  [\sigma]$. By Chern--Weil theory, one can find an $S^1$-connection one-form $\lambda_\pre $ on $Y$ with curvature $\dd\lambda_\pre  = \sigma$. One concludes $(Y,\lambda_\pre)$ is a contact manifold and we call $\xi_\pre = \ker \lambda_\pre$ the \emph{prequantization contact structure}. The contact form $\lambda_\pre$ is normalized so that it evaluates to one on the vector field generating the $S^1$-action. Thus the Reeb flow of $(Y,\lambda_\pre)$ is totally periodic in the fibres, with period one.

Returning to the above context, the Liouville domain $(M_{\max},\dd\lambda)$ will have contact boundary $(Y,\lambda)$, where $Y$ is an $S^1$-bundle over $\Sigma_k$. Because $(Y,\lambda)$ bounds a symplectic disk bundle, $\dd\lambda|_Y$ must agree with the pullback of an area form from $\Sigma_k$. Hence $\ker(\dd\lambda|_Y) = \ker(\dd{p} :TY\to T\Sigma)$ and the Reeb flow of $\lambda$, which is tangent to $\ker(\dd\lambda)$, must be along the fibres of the circle bundle. It follows that $(Y,\lambda)$ is the prequantization structure and $\lambda$ is some constant rescaling of $\lambda_\pre$. Note $Y$ is the $S^1$-bundle over $\Sigma_k$ with Euler number 
\[e(Y) = -[\Sigma_k]^2= -\langle k[\w] , [\Sigma_k]\rangle.\] 
Since $\dd\lambda =\w|_Y$, in order that the $\dd\lambda_\pre$ coincides with the pullback to $Y$ of a representative for the Euler class, the Liouville one-form $\lambda$ on $M_{\max}$ must be $\frac{1}{k} \lambda_{\pre}$ (see also \cite[Prop.\!\! 5]{Gir}).

\begin{example}Let us give this Biran decomposition explicitly in a few cases.
	\begin{enumerate}
		\item Consider $(\C P^2 ,\w_{FS})$, with Fubini--Study form scaled so that $\w_{FS}([\C P^1])=1$. Then a hyperplane $\C P^1 \subset \C P^2$ is a Donaldson divisor Poincar\'e dual to $[\w_{FS}]$. And $\C P^2$ decomposes into a 4-ball and a symplectic disk bundle over $\C P^1=S^2$ of Euler class $[\C P^1]^2 = 1$. The two are glued along the standard tight contact $S^3$, which is the prequantization circle bundle over $S^2$ of Euler number $-1$. 
		\item Consider $(S^2\times S^2,\w)$, where $\w$ is the product symplectic form so that each $S^2$ factor has area one. The diagonal $\Delta_{S^2}$ is a Donaldson divisor Poincar\'e dual to $[\w]$. We can decompose $S^2\times S^2$ as a cotangent disk bundle $D^\ast S^2$, given as a tubular neighbourhood of the anti-diagonal Lagrangian $S^2$, and the symplectic disk bundle over $S^2$, given as a tubular neighbourhood of $\Delta_{S^2}$. The two are glued along the prequantization circle bundle over $S^2$ of Euler number $-2 = -[\Delta_{S^2}]^2$.  
		\item Return to $(\C P^2,\w_{FS})$ and let $\phi: \C P^1 \to \C P^2$ be an embedded elliptic curve. Then $\Sigma =\im(\phi)$ is a genus one surface representing $3[\C P^1]$ in homology, and hence also a Donaldson divisor. The ample divisor complement $\C P^2 \setminus \Sigma$ is an affine variety, and as such has a natural Stein structure. Its ideal boundary is the prequantization circle bundle over $T^2$ with Euler number $ -9= -[\Sigma]^2$.
	\end{enumerate}
\end{example}

\begin{lemma}\label{lem:preq}
	Let $(Y,\lambda_{\pre})$ be the prequantization contact 3-manifold on a circle bundle with Euler number $-e$ over an oriented surface of genus $g$.
	\begin{enumerate}
		\item The first element of the ECH spectrum satisfies $c_1(Y,\lambda_{\pre}) \geq e$.
		\item If $e\leq 2g-2$, then $c_1(Y,\lambda_{\pre})=\infty$.
	\end{enumerate}
\end{lemma}
\begin{proof}
	From the spectrality property, elements of the ECH spectrum are always equal to the action of some null-homologous orbit sets of the Reeb flow. From the Gysin sequence, $H_1(Y) = \Z^{2g} \oplus \Z/e\Z$, where the circular fibre generates the $\Z/e\Z$ summand, and the rest of the homology comes from the base.
	
	In the case of $(Y,\lambda_{\pre})$, every simple orbit is just some circular fibre. Hence, for an orbit set to be null-homologous, it must have total multiplicity at least $e$. Since each simple orbit has period one, a null-homologous orbit set has action at least $e$. It follows that $c_1(Y,\lambda_{\pre}) \geq e$.
	
	The fact $c_1(Y,\lambda_{\pre})=\infty$ when $e\leq 2g-2$ was proven in \cite[Cor.\!\! 4.7]{Bei}. The basic idea is that when $e\leq 2g-2$, the Milnor--Wood inequality implies $(Y,\lambda_\pre)$ admits a foliation transverse to its Reeb flow. This foliation is necessarily hypertaut, and as such, a modification of a construction of Eliashberg--Thurston \cite{ET} gives rise to a Liouville domain of the form $I\times Y$ with $(Y,\lambda_\pre)$ as one boundary component. The main result of \cite[Thm.\!\ 1.2]{Bei} then implies the ECH spectrum must be infinite\end{proof}
\begin{remark}\label{rmk:preq}
	Based on some comparisons with Heegaard--Floer homology and some computations we were informed of in private communication with Michael Hutchings, we expect $c_k(Y,\lambda_\pre)<\infty$ for all $k$  whenever $e>2g-2$. This is known for all prequantization bundles over the sphere and torus thanks to Chen \cite{Che}, and further recent work of Chen \cite{CheU} implies the remaining cases.
\end{remark}

\begin{lemma}\label{lem:Dondisc}
	Suppose that $(X,\w)$ is a closed integral symplectic 4-manifold which admits a disconnected Donaldson divisor. Then its ECH capacities are infinite.
\end{lemma}
\begin{proof}
	According to the Biran decomposition \eqref{eq:Biran}, $(X,\w)$ will admit an embedded Liouville domain $(M,\dd\lambda)$ with disconnected boundary. Then \cite[Thm.\!\ 1.2]{Bei} will imply the ECH capacities of $(M,\dd\lambda)$ are infinite. Hence, by monotonicity, the ECH capacities of $(X,\w)$ are infinite too.
\end{proof}
\begin{remark}
	One can find disconnected Donaldson divisors, for example, in the standard symplectic $T^4$ \cite[Prop.\!\! 9]{Gir}. However, a version of the light cone lemma implies disconnected Donaldson divisors cannot exist in a symplectic 4-manifold with $b_2^+=1$.
\end{remark}
	
\begin{proposition}\label{prop:closed_rat_inf}
	Let $(X,\w)$ be a closed rational symplectic 4-manifold. Then its ECH capacities are all infinite.
\end{proposition}
\begin{proof}
	By rescaling, suppose $(X,\w)$ is integral. Let $k$ be a sufficiently large positive integer so that we may find a Donaldson divisor $\Sigma_k$ Poincar\'e dual to $k[\w]$. By \Cref{lem:Dondisc}, it suffices to consider the case where $\Sigma_k$ is connected for all $k$. This gives rise to a Biran decomposition as in \eqref{eq:Biran}. We can take the Liouville piece to be a rescaling of the maximal Liouville compactification $(M_{\max},\dd\lambda)$ of $X\setminus \Sigma_k$ by a constant smaller than, but arbitrarily close to, one. By conformality and monotonicity of the ECH capacities, it follows that
	\begin{equation}\label{eq:Biran_bound}
		c_1^\ECH(M_{\max},\dd\lambda) \leq c_1^\ECH(X,\w) .
	\end{equation}
	As discussed above, the boundary of $(M_{\max}, \dd\lambda)$ is the prequantization contact manifold $(Y_{\pre}, \frac{1}{k} \lambda_{\pre})$. Here, $Y_{\pre}$ is the $S^1$-bundle over $\Sigma_k$ with Euler number
	\begin{equation}\label{eq:Euler}
		e = -[\Sigma_k]^2 = -\langle k [\w]\cup k[\w], [X]\rangle = -2k^2 \vol(X,\w).
	\end{equation}
	Now by definition of the ECH capacities and their conformality properties, we see
	\begin{align*}
		c_1^\ECH(M_{\max},\dd\lambda) &=  \frac{1}{k} c_1(Y,\lambda_{\pre}).
		\shortintertext{In light of \Cref{lem:preq} and \eqref{eq:Euler},}
		&\geq 2k \vol(X,\w).
	\end{align*}
	Combined with \eqref{eq:Biran_bound}, we have $c_1^\ECH(X,\w) \geq 2k \vol(X,\w)$. But $k$ is allowed to be arbitrarily large, giving the desired result. 
\end{proof}	
	
	Now we consider the case where $(X,\w)$ is not rational. For this we will use ``singular Donaldson divisors," as introduced by Opshtein \cite{Ops}. By barycentric approximation, we can decompose 
	\[\w = \sum_{i=1}^\ell \lambda_i \w_i, \quad \lambda_i\in\R_{>0},\]
	where $\w_i$ are closed 2-forms with $[\w_i]\in H^2(X;\Q)$ and so that all $\w_i$ are within some $\ep$-ball of $\w$ in $H^2(X;\R)$ for any given $\ep>0$. We can also assume the $[\w_i]$ are linearly independent. When $\ep$ is sufficiently small, the $\w_i$ will all be symplectic forms.
	
	 We can then find positive integers $k_1,\ldots, k_\ell$ so that there is an $\w_i$-Donaldson divisor $\Sigma_i$ with $\mathrm{PD}[\Sigma_i] = k_i[\w_i]$ for each $i=1,\ldots,\ell$. After decreasing $\ep$, we can assume each $\Sigma_i$ is $\w$-symplectic as well. Moreover, following Opshtein \cite[Thm.\!\ 1.2]{Ops}, we can ensure the divisors intersect each other positively and transversely and are symplectically orthogonal with respect to $\w$. Also, there is some $N>0$ so that the $k_i$ can be chosen arbitrarily as long as they are all at least $N$, since one can find Donaldson divisors of any sufficiently large degree. By increasing $N$, we may also assume each $\Sigma_i$ is connected \cite{Gir}. Note we have
	\begin{equation}\label{eq:singpolar}
		[\w] = \sum_{i=1}^\ell a_i \mathrm{PD}[\Sigma_i], \quad a_i\in \R_{>0}. 
	\end{equation}
	Each $a_i = \lambda_i k_i^{-1}$. We will also assume for now that the union of the divisors $\Sigma_i$ has one connected component.

Given such a setup, let $U_i$ denote a small tubular neighbourhood of $\Sigma$ which can be chosen to be symplectomorphic to a symplectic disk bundle, as in the Biran decomposition. Let $\mathcal{U} = \bigcup_{i=1}^\ell U_i$. Topologically, $\mathcal{U}$ is a plumbing of disk bundles. The boundary of $\mathcal{U}$ is a graph manifold; it consists of $\ell$ many circle bundles of respective Euler classes $[\Sigma_i]^2$ over copies of $\Sigma_i$ (with added boundary components) glued along the boundary tori according to the intersection graph of the surfaces $\Sigma_i$. Symplectically, in the terminology of \cite{LM}, \eqref{eq:singpolar} implies the singular divisor satisfies the ``positive GS criterion" and hence $\mathcal{U}$ is a concave symplectic filling of its boundary.

Without any modification, the boundary of $\mathcal{U}$ is not smooth at the tori where the circle bundles meet, but can be made smooth by an arbitrarily small deformation of the neighbourhood. We will describe a semi-explicit such modification momentarily. More detailed versions of this construction appear for various kinds of singular symplectic divisors in many corners of the literature (e.g. \cite{GaS,McL,GP,ECHplumb}); since we will not require a toric structure\footnote{Such a structure will only exist under strong topological assumptions on the singular divisor \cite{ECHplumb}.} or non-degenerate/Morse--Bott Reeb dynamics on the boundary, precise formulae are not necessary.  Following \cite[Lem.\!\! 4.1]{Ops}, the divisor complement $X \setminus \bigcup_{i=1}^\ell \Sigma_i$ carries a complete Liouville vector field, and hence $(X\setminus \mathcal{U}, \w)$, after smoothing the boundary, is a Liouville domain.  
	
	Let $Y$ denote the (smoothed) boundary of $\mathcal{U}$. In the complement of a neighbourhood of the singularities, $Y$ is the union of boundaries of symplectic disk bundles. Following our discussion for the Biran decomposition, the characteristic foliation on $Y$ away from the singularities is periodic in the fibres of each circle bundle. For a sufficiently small radius of circle bundle, the flow with respect to the induced contact form has period arbitrarily close to $a_i$ on the fibres over $\Sigma_i$. 
	
	We now describe how to smooth $Y$ near the singularities and understand the resulting Reeb dynamics. By a Darboux--Moser--Weinstein construction as given in \cite[Prop.\!\! 3.1]{Ops}, the neighbourhood $(\mathcal{U},\w)$ can be chosen as a union of symplectic disk bundles with the following additional structure. Using symplectic orthogonality, for any intersection point $p_{ij} \in \Sigma_i \cap \Sigma_j$, there is a pair of small 2-disks $D_i, D_j$ with polar coordinates $(r_i,\theta_i)$ and $(r_j, \theta_j)$ so that a neighbourhood of $p_{ij}$ is identified with a bi-disk $D_i\times D_j$ centered on $p_{ij}$ and the symplectic fibrations over $\Sigma_i$ and $\Sigma_j$ are locally given by the projections onto $D_i$ and $D_j$ respectively. Finally, under this identification, the symplectic form $\w$ is given by
	\begin{equation}\label{eq:bidisk}
		\w|_{D_i\times D_j} = a_i \dd(r_i^2)\wedge \dd\theta_i +a_j\dd(r_j^2) \wedge \dd\theta_j.
	\end{equation}
	 In particular, the local model near an intersection point $p_{ij} \in \Sigma_i \cap \Sigma_j$ depends only on $a_i$ and $a_j$. To describe a smoothing of $Y$ near a singular torus determined by  $p_{ij}$, we can consider the simplest case of symplectic divisors $S^2 \times \{\ast\}$ and $\{\ast\}\times S^2$ in $S^2 \times S^2$ with the symplectic form $a_i \w_1 + a_j \w_2$, where $\w_1, \w_2$ are the pullbacks of normalized area forms from the two $S^2$ factors. On the complement of a codimension-one locus away from the divisors, this is symplectomorphic to the bi-disk \eqref{eq:bidisk} in an obvious way. With this symplectic form, the complement of a small symplectic disk neighbourhood of these divisors is the closed polydisk
	\[P(a_i-\ep,a_j-\ep)  = \{(z_1, z_2)\in \C^2  : \pi |z_1|^2 \leq a_i-\ep , \pi |z_2|^2 \leq a_j-\ep\} \]
	for $\ep>0$ arbitrarily small. The polydisk's boundary is singular precisely along the torus where 
	\[\pi|z_1|^2=a_i-\ep \qand  \pi|z_2|^2 = a_j-\ep.\]
	Under our preceding local symplectomorphism \eqref{eq:bidisk}, a neighbourhood of $Y$ around the singular torus corresponding to $p_{ij}$ is identified with a neighbourhood of the polydisk's boundary around its singular torus (for an appropriate $\ep$). 
	
	Let $\Omega$ denote a closed convex domain in the first quadrant of $\R^2$ which is very near to the rectangle $R=[0,a_i-\ep]\times [0,a_j-\ep]$, except that the domain is smoothed near $(a_i-\ep,a_j-\ep)$ to look like the undergraph of a concave decreasing function. See \Cref{fig:polydisk_smoothing}. We can consider the associated convex toric domain
	\[X_\Omega = \{(z_1, z_2)\in \C^2  : (\pi|z_1|^2, \pi |z_2|^2) \in \Omega\}. \]
	This will be a small deformation of the polydisk $P(a_i-\ep,a_j-\ep)$ with smooth boundary.	
	
\begin{figure}[h]
\begin{tikzpicture}
    \draw[->, thick] (-0.5, 0) -- (5.5, 0) node[above] {\small $\pi|z_1|^2$};
    \draw[->, thick] (0, -0.5) -- (0, 4.5) node[right] {\small $\pi|z_2|^2$};
    \filldraw[fill=green!15, draw=teal!80!black, thick] 
        (0,0) -- (4,0) -- (4,3) -- (0,3) -- cycle;
    \node at (2, 1.5) {\Large $R$};
    \node[below left] at (0,0) {\small $0$};
    \node[below] at (4,0) {\small $a_i-\ep$};
    \node[left] at (0,3) {\small $a_j-\ep$};
\end{tikzpicture}
\begin{tikzpicture}
    \draw[->, thick] (-0.5, 0) -- (5.5, 0) node[above] {\small $\pi|z_1|^2$};
    \draw[->, thick] (0, -0.5) -- (0, 4.5) node[right] {\small $\pi|z_2|^2$};
    \filldraw[fill=cyan!15, draw=blue!80!black, thick] 
        (0,0) -- (4,0) -- (4,2.7) arc (0:90:0.3) -- (0,3) -- cycle;
    \node at (2, 1.5) {\Large $\Omega$};
    \node[below left] at (0,0) {\small $0$};
    \node[below] at (4,0) {\small $a_i-\ep$};
    \node[left] at (0,3) {\small $a_j-\ep$};
\end{tikzpicture}
\caption{Toric base diagrams of the polydisk $P(a_i-\ep,a_j-\ep)$ on the left and its smoothing $X_\Omega$ on the right.}
\label{fig:polydisk_smoothing}
\end{figure}
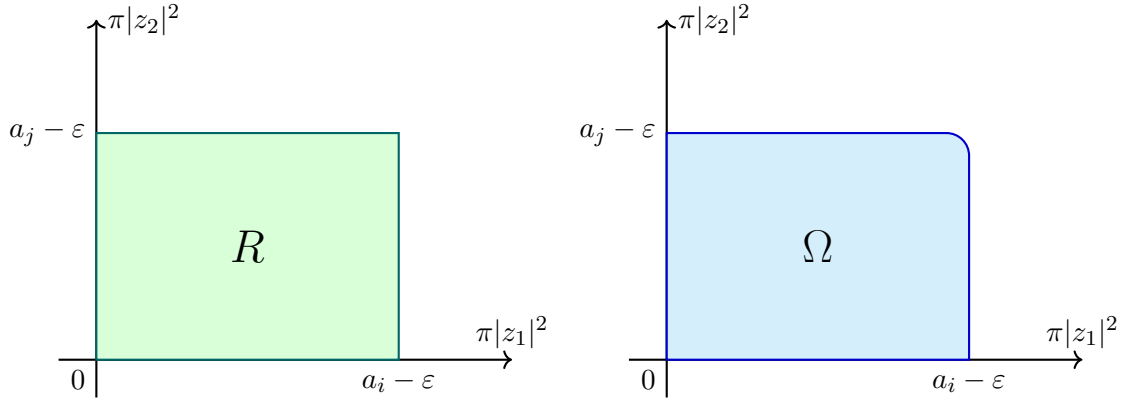
	
The Reeb dynamics of the boundary of a toric domain are well known; see e.g. \cite{CCGFHR, Hutnotes}. First, the Reeb flow will be tangent to the fibres of the toric projection
\[\Pi: \pd X_\Omega \subset \C^2 \to \pd\Omega \subset \R^2 , \quad \Pi(z_1,z_2) = (\pi |z_1|^2, \pi |z_2|^2). \]
Along $\Pi^{-1}$ of the horizontal edge, the Reeb flow will be periodic with period $a_j-\ep$. Along $\Pi^{-1}$ of the vertical edge, the Reeb flow will be periodic with period $a_i-\ep$. At a point $p$ on the curved part of the boundary of $\Omega$ with a tangent line of slope $s$, $\Pi^{-1}(p)$ defines an invariant torus. If $s$ is irrational, this is an irrational rotation of the torus. If $s$ is rational and can be written as $s= - \frac{v_1}{v_2}$ for $v_1,v_2\in \N$ relatively prime, then the torus is foliated by a circle of closed Reeb orbits each of period slightly smaller than $v_1 (a_i-\ep) + v_2 (a_j-\ep)$. 
	
	\begin{figure}[h]
	\includegraphics{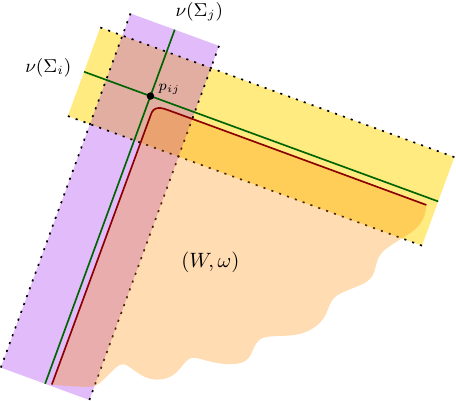}	
	\caption{Two irreducible components $\Sigma_i,\Sigma_j$ of a singular divisor intersect at $p_{ij}$. They admit symplectic disk neighbourhoods $\nu(\Sigma_i),\nu(\Sigma_j)$. In the bi-disk $\nu(\Sigma_i)\cap \nu(\Sigma_j)$, we can use the previous toric model to smooth the complement of smaller symplectic neighbourhoods $\mathcal{U} = \bigcup U_i$ to a Liouville domain $(W,\w)$.}
	\label{fig:divcomp}		
\end{figure}

We can now deform the divisor complement $(X\setminus \mathcal{U}, \w)$ by replacing a neighbourhood of the singular intersection (identified with a neighbourhood of the corner of a polydisk) with a neighbourhood of the domain $X_{\Omega}$ near the smoothed corner, as in \Cref{fig:divcomp}. As before, this determines an embedded Liouville domain inside $(X,\w)$. We also know explicitly the Reeb flow on its boundary $Y$. To describe the Reeb orbits fully, it is helpful to first determine the homology of $Y$.

	\begin{lemma}\label{lem:graph_homology}
		Let $\Gamma$ be a connected graph of order $\ell$, with an ordering of vertices, whose $i$th vertex is labelled by a pair of non-negative integers $(g_i, e_i)$ for $i=1,\ldots, \ell$. Let $b_1(\Gamma)$ denote its first Betti number, $g_{\mathrm{tot}}$ denote the sum of the labels $g_i$, and $Q_\Gamma$ the $\ell\times \ell$ matrix which is the sum of the adjacency matrix of $\Gamma$ and the diagonal matrix $\mathrm{diag}(e_1,\ldots, e_\ell)$. 
		
		Let $X_\Gamma$ be the 4-manifold obtained by plumbing disk bundles of Euler number $e_i$ over surfaces of genus $g_i$ for $i=1,\ldots, \ell$ according to the edges of the graph $\Gamma$. Let $Y_\Gamma$ be the boundary of $X_\Gamma$. Then
		\[H_1(Y_\Gamma;\Z) \cong \Z^{2g_\mathrm{tot}+b_1(\Gamma)} \oplus \coker(Q_\Gamma : \Z^\ell \to \Z^\ell). \]
	\end{lemma}
	\begin{proof}
	By contracting the disk fibres, $X_\Gamma$ deformation retracts to a collection of surfaces $\Sigma_1,\ldots, \Sigma_\ell$ of genera $g_1,\ldots, g_\ell$ joined by one-handles according to the edges of the graph $\Gamma$. It then follows that
	\[H_1(X_\Gamma) \cong \Z^{2g_\mathrm{tot}} \oplus \Z^{b_1(\Gamma)},\quad H_2(X_\Gamma) \cong \Z^\ell, \qand H_3(X_\Gamma)=H_4(X_\Gamma)=0.\]
	From Poincar\'e--Lefschetz duality and the universal coefficient theorem, we find
	\[H_2(X_\Gamma, \pd X_\Gamma) \cong \Z^\ell ,\quad  H_3(X_\Gamma, \pd X_\Gamma)\cong \Z^{2g_\mathrm{tot}+b_1(\Gamma)},\qand H_1(X_\Gamma,\pd X_\Gamma) =0. \]
	The pair $(X_\Gamma, \pd X_\Gamma= Y_\Gamma)$, yields the long exact sequence fragment
	\[ H_2(X_\Gamma) \xrightarrow{i_\ast} H_2(X_\Gamma,Y_\Gamma)\to H_1(Y_\Gamma) \to H_1(X_\Gamma)\to 0.\]
	We claim that, with the correct identifications of the domain/codomain with $\Z^\ell$, the first map $i_\ast$ is given by $Q_\Gamma$. The lemma then follows because the resulting short exact sequence 
	\[0\to \coker(Q_\Gamma) \to H_1(Y_\Gamma) \to \Z^{2g_\mathrm{tot}+b_1(\Gamma)}\to 0\]
	ends in a free module and so must split.
	
	To see our claim, we can compose with isomorphisms to yield a map
	\[H_2(X_\Gamma) \xrightarrow{i_\ast} H_2(X_\Gamma, \pd X_\Gamma) \xrightarrow{\cong} H^2(X_\Gamma) \xrightarrow{\cong} \Hom(H_2(X_\Gamma), \Z),\]
	where the second map is Poincar\'e--Lefschetz duality and the third is the universal coefficient theorem. Let $\Sigma_i,\Sigma_j \subset X_\Gamma$ be a pair of surfaces forming the base of two disk bundles in the plumbing (suitably perturbed if $i=j$). We compute
	\begin{align*}
		\langle \mathrm{PD}(i_\ast [\Sigma_i]),[\Sigma_j]\rangle &= \big( [X_\Gamma,\pd X_\Gamma]\cap \mathrm{PD}(i_\ast [\Sigma_i]) \big) \cap \mathrm{PD}[\Sigma_j]\\
		&= [X_\Gamma,\pd X_\Gamma] \cap \big(\mathrm{PD}(i_\ast [\Sigma_i]) \cup \mathrm{PD}[\Sigma_j]\big).
		\shortintertext{From intersection theory,}
		&= \#(\Sigma_i \cap \Sigma_j).
	\end{align*}
	Evidently this intersection count is the Euler number of the bundle over $\Sigma_i$ if $i=j$ and is the number of plumbings between $\Sigma_i$ and $\Sigma_j$ if $i\neq j$. Equivalently, this is the $(i,j)$th entry of $Q_\Gamma$. Because the surfaces $\Sigma_i$ for $i=1,\ldots, \ell$ generate $H_2(X_\Gamma)$, the claim follows.
	\end{proof}
	We may apply \Cref{lem:graph_homology} to the boundary $Y$ of our Liouville domain $X\setminus \mathcal{U}$. Let $\Gamma$ be the labelled plumbing graph of our singular divisor configuration. Consider the class of each irreducible component $[\Sigma_i]\in H_2(\mathcal{U}) \cong \Z^\ell$ and let $\theta_i$ denote its image in $\coker(Q_\Gamma) \subset H_1(Y_\Gamma)$. Our preceding analysis gives the following enumeration of Reeb orbits on $Y=Y_\Gamma$. Compare a similar description of the Reeb dynamics of the boundary of ample divisor complements given in all dimensions by Ganatra--Siegel \cite[Thm.\!\! 2.7]{GS} (the two descriptions would agree after we perturb our Reeb flow to be non-degenerate).
	\begin{proposition}\label{prop:div_comp_Reeb}
		Fix $\delta>0$. There exists a tubular neighbourhood $\mathcal{U}$ of the singular divisor $\bigcup_{i=1}^\ell \Sigma_i$ whose complement is a Liouville domain $(X\setminus \mathcal{U}, \w)$ with contact boundary $(Y_\Gamma, \lambda_\delta)$ so that every simple closed Reeb orbit $\gamma$ of $(Y_\Gamma, \lambda_\delta)$ has the following form. There is a pair $1\leq i < j \leq \ell$ and a pair of non-negative integers $v_1, v_2$, not both zero and relatively prime if both positive, so that $\gamma$ represents the homology class
		\[[\gamma] = v_1 \theta_i + v_2 \theta_j \in H_1(Y_\Gamma) \]
		and has period within $\delta$ of $v_1 a_i + v_2 a_j$.\qed
	\end{proposition}

\begin{lemma}\label{lem:sing_div_c1_est}
	Let $(Y_\Gamma, \lambda_\delta)$ be the contact-type hypersurface given in the statement of \Cref{prop:div_comp_Reeb}. Suppose $K\in \N$ divides all the entries of the matrix $Q_\Gamma$. Then the first element of the ECH spectrum satisfies
	\[c_1(Y_\Gamma, \lambda_\delta)\geq K\cdot \Big(\min_{i=1,\ldots, \ell}(a_i)-\delta\Big). \]
\end{lemma}
\begin{proof}
	By spectrality, $c_1(Y_\Gamma, \lambda_\delta)$ must equal the action of some null-homologous orbit set of the Reeb flow (our Reeb flow is degenerate, but spectrality continues to hold by Arzel\`a--Ascoli). It then follows from \Cref{lem:graph_homology} and \Cref{prop:div_comp_Reeb} that
	\begin{equation}\label{eq:c1_graph_bound}
	\begin{split}
		c_1(Y_\Gamma, \lambda_\delta) &\geq \min \bigg\{\sum_{i=1}^m \mathcal{A}(\gamma_i) \,\bigg|\, \text{$\gamma_1,\ldots,\gamma_m$ Reeb orbits},\, \sum_{i=1}^m [\gamma_i]=0\in H_1(Y_\Gamma) \bigg\}\\
		&\geq  \min \left \{\sum_{i=1}^\ell v_i(a_i-\delta) \,\bigg|\, \vec{v}=(v_1,\ldots, v_\ell) \in \Z_{\geq0}^\ell,\, \vec{v} \in \im(Q_\Gamma) \setminus \{0\}\right\}. 
	\end{split}
	\end{equation}
	Observe that if $K$ divides all the entries of $Q_\Gamma$, then $Q_\Gamma$ is $K$ times an element of $\GL(\ell,\Z)$ and hence the image of $Q_\Gamma$ lands in the lattice $(K\Z)^\ell \subset \Z^\ell$. From \eqref{eq:c1_graph_bound}, we find
	\begin{align*}
		c_1(Y_\Gamma, \lambda_\delta) &\geq \min \left \{\sum_{i=1}^\ell v_i(a_i-\delta) \,\bigg|\, \vec{v} \in (K\Z_{\geq 0})^\ell,\, \vec{v} \neq 0\right\}\\
		&\geq K\cdot \Big(\min_{i=1,\ldots, \ell}(a_i)-\delta\Big). \qedhere
	\end{align*}
	\end{proof}

\begin{proof}[Proof of \Cref{thm:closed_inf}]
Let $(X,\w)$ be a closed symplectic 4-manifold. If the symplectic form is rational, the ECH capacities of $(X,\w)$ are infinite by \Cref{prop:closed_rat_inf}.

If the symplectic form is not rational, consider a singular Donaldson divisor as described in the preceding pages. This yields a decomposition $\w = \sum_{i=1}^\ell \lambda_i \w_i$ for rational symplectic forms $\w_i$ and $\lambda_i>0$. For any $k_1,\ldots, k_\ell$ all sufficiently large, we obtain a collection of connected closed embedded symplectic submanifolds $\Sigma_1,\ldots, \Sigma_\ell$ which intersect transversely, positively, and $\w$-orthogonally so that $\mathrm{PD}[\Sigma_i] = k_i \w_i$. If the union of these divisors has more than one connected component, then the same proof as for \Cref{lem:Dondisc} implies the ECH capacities of $(X,\w)$ are infinite. Thus we may also assume the union of the divisors is connected.

Fix $K\in\N$ sufficiently large. Choose the degrees $k_i$ so that for each $i$,
\begin{equation}\label{eq:div_div}
	[\Sigma_i] \in K\cdot H_2(X;\Z) \subset H_2(X;\Z).
\end{equation}
With the forms $\w_1,\ldots, \w_\ell$ fixed, let $L$ be large enough so that $L\w_i$ lives in the integral lattice for each $i$ and there is a Donaldson divisor dual to $tL\w_i$ for any $t\in\N$. We can then assume the $k_i$ for which \eqref{eq:div_div} is satisfied are chosen so $k_i \leq K L$ for each $i$, regardless of $K$.

By \Cref{prop:div_comp_Reeb}, for any $\delta>0$, we obtain an embedded Liouville domain $(W,\w) \hookrightarrow (X,\w)$ in the complement of a tubular neighbourhood of the surfaces $\Sigma_i$. The boundary of this domain is the contact manifold $(Y_\Gamma, \lambda_\delta)$, where $\Gamma$ is the vertex-labelled graph associated to the singular configuration of divisors.  Let $Q_\Gamma$ be the associated intersection matrix, as in \Cref{lem:graph_homology}. Note the $(i,j)$th entry of $Q_\Gamma$ is exactly
\[(Q_\Gamma)_{ij} = [\Sigma_i]\cdot [\Sigma_j].\]
Because of the condition \eqref{eq:div_div}, each entry $(Q_\Gamma)_{ij}$ is divisible by $K^2$. Accordingly, \Cref{lem:sing_div_c1_est} applies to give
\[c_1(Y_\Gamma, \lambda_\delta) \geq K^2 \min_{i=1,\ldots, \ell} (\lambda_i k_i^{-1}) -K^2\delta.\]
Note by definition and monotonicity of the ECH capacities,
\[c_1(Y_\Gamma, \lambda_\delta) = c_1^\ECH(W,\w) \leq c_1^\ECH(X,\w).\]
This is true for $\delta$ arbitrarily small, so we deduce
\begin{align*}
c_1^\ECH(X,\w) &\geq K^2 \min_{i=1,\ldots, \ell} (\lambda_i k_i^{-1}).
\shortintertext{Since we took $k_i \leq KL$,}
&\geq  \frac{K}{L}\min_{i=1,\ldots, \ell} (\lambda_i).	
\end{align*}
But recall we could take $K$ arbitrarily large for $L$ and $\lambda_1,\ldots, \lambda_\ell$ fixed and positive. Hence $c_1^\ECH(X,\w)$ must be infinite.
\end{proof}

We now prove \Cref{thm:posinf} via the following more precise result.
\begin{theorem}\label{thm:posinf_div}
	Suppose $(X,\w)$ is a closed, rational symplectic 4-manifold satisfying 
\begin{equation}\label{eq:c1_cond}
	c_1(TX,\w)\cdot [\w] \leq 0.
\end{equation}
Let $\Sigma$ be a closed embedded symplectic surface so that $\mathrm{PD}[\Sigma] = k[\w]\in H^2(X;\Z)$ for some $k>0$. Then the complement of a small tubular neighbourhood of $\Sigma$ is a Liouville domain whose ECH capacities are infinite.
\end{theorem}
Note this implies \Cref{thm:posinf} since we can apply \Cref{thm:posinf_div} to any Donaldson divisor.
\begin{proof}
Suppose we are given $(X,\w)$ closed, rational, and satisfying \eqref{eq:c1_cond}. Note if the theorem is proven for a symplectic manifold $(X,\w)$, it is also proven for any rescaling of $\w$, since a surface being symplectic is independent of scaling, and we can rescale the embedded Liouville domain and use monotonicity of the capacities. Thus, without loss of generality, we may rescale $\w$ to be integral; \eqref{eq:c1_cond} will continue to hold. 

Now $\Sigma$ can be regarded as a Donaldson divisor for $(X,\w)$ and the complement of a tubular neighbourhood of $\Sigma$ is Liouville isotopic to the Liouville domain piece of the Biran decomposition, as explained above. Call this domain $(M,\dd\lambda)$.

If $\Sigma$ is disconnected, \Cref{lem:Dondisc} already implies the ECH capacities of $(M,\dd\lambda)$ are infinite. Otherwise, since $\Sigma$ is a connected closed embedded symplectic surface in $(X,\w)$ of some genus $g$, the adjunction formula dictates that
		\begin{equation}\label{eq:adj}
			2g-2 = [\Sigma]^2 - \langle c_1(TX,\w),  [\Sigma]\rangle.
		\end{equation}
	Since $[\Sigma]$ is Poincar\'{e} dual to $k[\w]$,  \eqref{eq:c1_cond} in conjunction with \eqref{eq:adj} gives
	\begin{align*}
		[\Sigma]^2 &=2g-2 +k\cdot  c_1(TX,\w)\cdot [\w]\\
		 &\leq 2g-2.
	\end{align*}
	The boundary of $(M,\dd\lambda)$ in the Biran decomposition is the prequantization contact manifold $(Y,\xi_\pre)$ where $Y$ is the $S^1$-bundle of Euler number $-e = -[\Sigma]^2$ over the genus $g$ surface $\Sigma$. We have just computed $e\leq 2g-2$, and so \Cref{lem:preq} implies the ECH spectrum of $(Y,\xi_\pre)$ is infinite. So, by definition, the ECH capacities of $(M,\dd\lambda)$ are infinite. 
\end{proof}
\begin{example}\label{ex:bad_uniruled}
	It is worth enumerating exactly which rational symplectic 4-manifolds satisfy \eqref{eq:c1_cond}, and hence \Cref{thm:posinf} and \Cref{thm:posinf_div} apply to. Firstly, \eqref{eq:c1_cond} holds for anything with $b_2^+>1$ \cite{MS96} or anything which is not K\"ahler \cite{LN}. So we focus on K\"ahler 4-manifolds with $b_2^+=1$, which is equivalent to asking that the K\"ahler manifold have geometric genus $p_g=0$. By way of contrast, \Cref{thm:alt_fin} will imply the alternative ECH capacities are finite for all examples below.
\begin{enumerate}[label=\textbf{\arabic*.}, leftmargin = 0.8cm]
	\item If $X$ is anything K\"ahler with $p_g=0$ and Kodaira dimension greater than $-\infty$, then \eqref{eq:c1_cond} holds \cite{MS96}. This includes many well known K\"ahler surfaces such as the Enriques, Dolgachev, Barlow, and hyperelliptic surfaces.
	\end{enumerate}
	The only remaining symplectic 4-manifolds, for which \eqref{eq:c1_cond} might fail, are symplectic rational or blown-up symplectic ruled K\"ahler surfaces.
	\begin{enumerate}[resume]
	\item Let $X=S^2\times \Sigma_g$ be the product ruled surface with $g\geq 2$. Consider volume forms $\sigma,\tau$ on $S^2$ and $\Sigma_g$ respectively, which are normalized to have volume one. For any $a,b>0$, we have a product symplectic form $\w_{a,b}$ on $X$ given by 
	\[\w_{a,b} = a\cdot \mathrm{pr}_1^\ast\sigma+b\cdot \mathrm{pr}_2^\ast\tau.\]
	Note that
	\[\mathrm{PD}(c_1(TX,\w_{a,b}))= (2-2g)[S^2]+2[\Sigma_g].\]
	Hence,
	\[c_1(TX,\w_{a,b})\cdot \w_{a,b} = (2-2g)a+2b.\]
	Provided $a> \frac{b}{g-1}$ (i.e. the fibre of $X$ is not too small relative to the base), we see that $(X,\w_{a,b})$ will satisfy \eqref{eq:c1_cond}. A similar computation will apply for non-product ruled surfaces and their blow-ups.
	\item Consider the blow-up $X=\C P^2 \# 10 \overline{\C P^2}$. Recall that $\C P^2$ admits a full symplectic packing of 9 or more equi-sized balls. In particular, we may find a symplectic embedding
	\[\coprod_{i=1}^{10} B^4(\lambda)\hookrightarrow \mathrm{int}(B^4(1)) \hookrightarrow \C P^2(1)\] 
	provided $\lambda< \frac{1}{\sqrt{10}}$. We then obtain a symplectic form $\w_\lambda$ on $X$ by blowing up along these 10 embedded balls. Let $H$ denote the hyperplane class in $H^2(X)$ and $E_1,\ldots, E_{10}$ the classes of exceptional divisors of the blow-ups. One has
	\[[\w_\lambda] = H - \lambda(E_1+\cdots +E_{10}).\]
	Recall that
	\[c_1(TX,\w_\lambda) = 3H -(E_1+\cdots +E_{10}).\]
	Hence, 
	\[c_1(TX,\w_\lambda)\cdot [\w_\lambda] = 3-10\lambda.\]
	Choosing $\sqrt{10} > 10 \lambda \geq 3$ implies $(X,\w_\lambda)$ satisfies \eqref{eq:c1_cond}. This same construction works for any blow-up $\C P^2 \# N \overline{\C P^2}$ with $N\geq 10$. It also applies to blow-ups of ruled surfaces over $T^2$.
	\end{enumerate}
\end{example}

We now briefly justify \Cref{conj:posinf}. Recently, Mark--Tosun proved the following.

\begin{theorem}[{\cite[Thm. 1.7]{MT}}]
	Let $(X,\w)$ be a closed symplectic 4-manifold with 
	\[c_1(TX,\w)\cdot [\w]>0.\]
	Let $(W,\w|_W)$ be a Weinstein domain embedded in $(X,\w)$. Then, possibly after shrinking $W$ under the reverse Liouville flow for finite time, one can find a closed connected embedded symplectic surface $\Sigma \subset X$ of genus $g$ which avoids the shrunken domain and satisfies
	\begin{equation}\label{eq:MTcond}
		[\Sigma]^2 > \max\{0, 2g-2\}. 
	\end{equation}
\end{theorem}

The proof of this theorem proceeds by first perturbing $\w$ to be rational, and then finding a Donaldson divisor for the perturbed symplectic form and showing it can be taken to avoid the isotropic skeleton of $(W,\w)$. The condition \eqref{eq:MTcond} follows as in our proof of \Cref{thm:posinf} from the adjunction inequality and the assumption on $c_1(TX,\w)$.

In the situation of Mark--Tosun's theorem, let $(\tilde{Y},\tilde{\xi})$ denote the contact structure on the boundary of $(W,\w)$. Since $[\Sigma]^2>0$, the region $X\setminus (\Sigma \cup W)$ gives rise to an exact symplectic cobordism from $(Y,\xi_{\pre})$ to $(\tilde{Y},\tilde{\xi})$, where $Y$ is the circle bundle over a genus $g$ surface with Euler number $-[\Sigma]^2$. In light of \Cref{rmk:preq}, the condition \eqref{eq:MTcond} implies the ECH spectrum of $(Y,\xi_{\pre})$ is all finite. By functoriality properties of ECH under cobordisms, this implies the ECH spectrum of $(\tilde{Y},\tilde{\xi})$ is finite and hence the ECH capacities of $(W,\w)$ are finite.

Thus, we have proved \Cref{conj:posinf} for Weinstein domains. If one could  extend Mark--Tosun's result to Liouville domains, then \Cref{conj:posinf} would follow.

\subsection{Zehnder tori, neck-stretching, and infinite alternative capacities}
This setion is devoted to a proof of \Cref{thm:alt_inf}. Consider a closed symplectic 4-manifold $(X,\Omega)$ admitting a hypersurface $Y \subset X$ which is a 3-torus with $S^1$-valued  coordinates $x_1,x_2,x_3$ so that
\[\w = \Omega|_Y =A_{12}\dd{x_1}\wedge\dd{x_2}+A_{13}\dd{x_1}\wedge\dd{x_3} +A_{23}\dd{x_2}\wedge\dd{x_2} \] 
and the triple of real numbers $(A_{12},A_{13},A_{23})$ is linearly independent over $\Q$. In particular, this means $A_{12}\neq 0$. Then we note
\[\dd{x_3} \wedge \w = A_{12} \dd{x_1}\wedge\dd{x_2}\wedge\dd{x_3}.\]
Hence $(\dd{x_3},\w)$ is a framed Hamiltonian structure on $Y$. Moreover, it is stable because $\dd{x_3}$ is closed. The associated  Hamiltonian vector field is
\[R = \frac{A_{23}}{A_{12}}\pd_{x_1} - \frac{A_{13}}{A_{12}}\pd_{x_2} + \pd_{x_3}.\]
Because of our assumption \eqref{eq:Zehhypcond}, this vector field defines an irrational rotation of the torus. The flow has no closed characteristics and all its trajectories are dense.

From the coistropic neighbourhood theorem for Hamiltonian hypersurfaces, we can find an open subset $i:U \subset X$ with $U= (-\varepsilon,\varepsilon)_t\times Y$, so that
\[i^\ast\Omega = A_{12}\dd{x_1}\wedge\dd{x_2}+A_{13}\dd{x_1}\wedge\dd{x_3} +A_{23}\dd{x_2}\wedge\dd{x_2} +\dd{t}\wedge\dd{x_3}.\]

We now pick three distinguished open sets $V_-,  W, V_+ \subset U$ and a distinguished point $p\in U$. We take
\[V_- = \left(-\frac{2\ep}{3},-\frac{\ep}{3}\right)\times Y, \quad W =\left(-\frac{\ep}{4},\frac{\ep}{4}\right)\times Y, \qand V_+ =\left(\frac{\ep}{3},\frac{2\ep}{3}\right)\times Y.\]
Fix some point in $W$ as our distinguished point $p$.

Given some $\Omega$-compatible almost-complex structure $J$ on $X$, we can consider the moduli space $\mathcal{M}^J(X,\Omega; p)$ of closed non-constant $J$-holomorphic curves constrained to pass through the point $p$. 

\begin{lemma}\label{lem:fredholm}
	Given any $\Omega$-compatible $J$, there is an $\Omega$-compatible $\tilde{J}$ which differs from $J$ only on $W$ and a point $p\in W$ so that any $u
	\in \mathcal{M}^{\tilde{J}}(X,\Omega;p)$ does not have its image contained in $U$.
\end{lemma}
\begin{proof}
	Let $\mathcal{M}_s^J(X,\Omega; p)\subset \mathcal{M}^J(X,\Omega; p)$ denote the subspace of irreducible somewhere-injective curves. Given any curve $u\in \mathcal{M}^J(X,\Omega; p)$, one can look at the underlying simple curve for each of its irreducible components. If $u$ is contained in $U$, then so will each of its simple components. Thus it suffices to prove the lemma for $u\in\mathcal{M}_s^{J}(X,\Omega; p)$.
	
	By standard transversality methods, for a Baire-generic almost-complex structure $\tilde{J}$ which agrees with $J$ outside of $W$, and a generic point $p$, the moduli space of $\tilde{J}$-holomorphic curves with a marked point passing through $p$ and a somewhere injective point in $W$ is a smooth manifold and each of its connected components has the correct expected dimension (see e.g. \cite[Thm. 4.6.1]{Wendl}).
	
	Since $p\in W$, every curve in $\mathcal{M}_s^{\widetilde{J}}(X,\Omega; p)$ passes though $W$. And since we consider simple curves, and closed simple curves have a dense collection of somewhere injective points, every curve in $\mathcal{M}_s^{\tilde{J}}(X,\Omega; p)$ has a somewhere injective point in $W$. So for generic $p$ and $\tilde{J}$, near each $u\in \mathcal{M}_s^{\tilde{J}}(X,\Omega; p)$, the moduli space is a smooth manifold with local dimension $\ind(u)-2$. Here, we must subtract two because asking a curve pass through a point in a 4-manifold is a codimension two constraint.
	
	Fix such a generic pair of $p$ and $\tilde{J}$. Suppose $u\in \mathcal{M}_s^{\tilde{J}}(X,\Omega; p)$. Let the homology class represented by this curve be $[u]=A \in H_2(X;\Z)$. Recall the Gromov--Taubes index
	\[I(A) = A\cdot A +\langle c_1(TX,\Omega),A\rangle.\]
	From the adjunction formula, we have $I(A)\geq \ind(u)$ with equality if and only if $u$ is embedded. Because $u$ lives in a manifold of dimension $\ind(u)-2$, in order that $u$ exists at all, it must be that $\ind(u)\geq 2$. Hence $I(A)\geq 2$.
	
	We claim that $I(A)=0$ if $u$ is contained in the open set $U$, which will conclude the proof. Let $j: T^3 \hookrightarrow X$ denote the inclusion into $X$ of the fibre $\{0\} \times Y$ of $U$. If $u$ is contained in $U$, then $A$ must lie in the image of the map 
	\[j_\ast: H_2(T^3)\cong H_2(U) \to H_2(X).\]
	Suppose $A = j_\ast \tilde{A}$ for $\tilde{A}\in H_2(T^3)$. We have
	\[\langle c_1(TX,\Omega), A\rangle = \langle c_1(TX,\Omega), j_\ast\tilde{A} \rangle = \langle j^\ast c_1(TX,\Omega), \tilde{A}\rangle = \langle c_1 (TX|_{Y},\Omega),\tilde{A}\rangle.\]
	This vanishes because the tangent bundle of $X$ has a symplectic trivialization over $U$.
	
	Because $\im(j) \subset X$ is displaceable by a translation, $\im(j_\ast) \subset H_2(X)$ is an isotropic subspace with respect to the intersection form; in particular, $j_\ast \alpha\cdot j_\ast \alpha=0$ for any $\alpha \in H_2(T^3)$. Hence, if $u$ is contained in $U$, then $A\cdot A=0$ and so $I(A)=0$.
\end{proof}

\begin{proof}[Proof of \Cref{thm:alt_inf}] Fix some $(\lambda,\w)$-compatible almost-complex structures $J_\pm$ on the symplectization regions $V_\pm$. We claim one can smoothly extend these to an $\Omega$-compatible almost-complex structure $J$ on $X$. First, the pair $J_\pm$ naturally extends to the closure $\overline{V_-\cup V_+}$. The space of compatible complex automorphisms of $\R^4$ is isomorphic to the homogeneous space $F=\Sp(4)/\U(2)$. It is well known that this space is contractible. Thus, an $\Omega$-compatible almost-complex structure on $X$ is the same as a section of a certain fibre bundle over $X$ with contractible fibre $F$. We wish to extend a given section on $\overline{V_-\cup V_+}$ to all of $X$. By a standard application of obstruction theory, such an extension always exists because $F$ is contractible. 

We will modify $J$ via a sequence of ``neck-stretched" almost-complex structures. This procedure is described in many places, e.g. see \cite[\S1.3]{EGH} for the original reference or Cieliebak--Mohnke \cite[\S2.7]{CM} for the generality of SHSs. Neck-stretching and its effect on the holomorphic curves we will study is sketched in \Cref{fig:neck-stretch}.

Consider some strictly increasing diffeomorphism $\phi_k: [-k-\ep, k+\ep] \to [-\ep, \ep]$ which is linear with derivative one on $[-k-\ep, -k-\frac{2\ep}{3}]$ and $[k+\frac{2\ep}{3}, k+\ep]$ and restricts to the identity $[-\frac{\ep}{3}, \frac{\ep}{3}] \to [-\frac{\ep}{3}, \frac{\ep}{3}]$, as in \Cref{fig:neckdiffeo}. Moreover, choose $\phi_k$ so that they $C^\infty$-converge on compact sets to a smooth diffeomorphism $\phi_\infty: (-\infty,\infty) \to (-\frac{2\ep}{3},\frac{2\ep}{3})$. We can additionally make it so that pointwise $|\phi_k(t)|$ and $\phi'_k(t)$ are decreasing sequences. Using $\phi_k$, we obtain a symplectomorphism
\[\big((-k-\ep, k+\ep)_t\times Y, \w + \dd(\phi_k(t)\lambda)\big) \xrightarrow{\cong} \big((-\ep,\ep)\times Y,\w+\dd(t\lambda)\big)=(U,i^\ast\Omega).\]

\begin{figure}[h]
\begin{tikzpicture}
    \draw[->, thick] (-7.5, 0) -- (7.5, 0) ;
    \draw[->, thick] (0, -4) -- (0, 4);

    \begin{scope}[gray, dashed, thin]
        \draw (-6, 0) -- (-6, -3) -- (0, -3);
        \draw (-5, 0) -- (-5, -2) -- (0, -2);
        \draw (-1, 0) -- (-1, -1) -- (0, -1);
        \draw (1, 0) -- (1, 1) -- (0, 1);
        \draw (5, 0) -- (5, 2) -- (0, 2);
        \draw (6, 0) -- (6, 3) -- (0, 3);
    \end{scope}
\draw (-6, 0.1) -- (-6, -0.1) node[xshift=-4pt, above=4pt] {$-k-\ep$};
    \draw (-5, 0.1) -- (-5, -0.1) node[below=-1pt, xshift=-2pt] {$-k-\frac{2\ep}{3}$};
    \draw (-1, 0.1) -- (-1, -0.1) node[above=4pt, xshift=-4pt] {$-\frac{\ep}{3}$};
    
    \draw (1, 0.1) -- (1, -0.1) node[below=2pt, xshift=1pt] {$\frac{\ep}{3}$};
    \draw (5, 0.1) -- (5, -0.1) node[below=-1pt, xshift=2.5pt] {$k+\frac{2\ep}{3}$};
    \draw (6, 0.1) -- (6, -0.1) node[xshift=2pt, above=3pt] {$k+\ep$};

    \draw (0.1, -3) -- (-0.1, -3) node[right=2pt] {$-\ep$};
    \draw (0.1, -2) -- (-0.1, -2) node[right=2pt] {$-\frac{2\ep}{3}$};
    \draw (0.1, -1) -- (-0.1, -1) node[below=1pt, right=2pt] {$-\frac{\ep}{3}$};
    
    \draw (0.1, 1) -- (-0.1, 1) node[left=2pt] {$\frac{\ep}{3}$};
    \draw (0.1, 2) -- (-0.1, 2) node[left=2pt] {$\frac{2\ep}{3}$};
    \draw (0.1, 3) -- (-0.1, 3) node[left=2pt] {$\ep$};
    
    \draw[blue, thick] 
        (-6, -3) -- (-5, -2) 
        .. controls (-4.5, -1.5) and (-1.5, -1.5) .. (-1, -1)
        -- (1, 1)
        .. controls (1.5, 1.5) and (4.5, 1.5) .. (5, 2)
        -- (6, 3);
    \node[blue, right] at (2.5, 1.1) {$\phi_k$};
    \node[below right] at (0,0) {$0$};
\end{tikzpicture}
\begin{caption}{A choice for the function $\phi_k$.}\label{fig:neckdiffeo}
\end{caption}
\end{figure}

The preimages of $V_\pm$ under this symplectomorphism are stretched regions $V^k_\pm$, identified as
\[V^k_- = \left(-k-\frac{2\ep}{3},-\frac{\ep}{3}\right)\times Y, \qand V^k_+ =\left(\frac{\ep}{3},k+\frac{2\ep}{3}\right)\times Y.\]

 One can remove $U$ from $X$ and replace it with $(-k-\ep, k+\ep)\times Y$ to produce a new manifold $(X_k,\Omega_k)$. This will be symplectomorphic to $X$, but preferable for our purposes.

The almost-complex structure $J$ on $X$ determines one $J^k$ on $X_k$. This is given by extending the almost-complex structures $J_\pm$ on $V_\pm$ to $J_\pm^k$ on the stretched symplectizations $V_\pm^k$ so that they remain translation invariant and admissible; away from $V^k_\pm$, $J$ and $J^k$ agree. The almost-complex structure $J^k$ is $\Omega_k$-compatible; the net effect of this procedure, after applying the symplectomorphism, is to keep the symplectic structure the same but alter the compatible almost-complex structure so that the associated metric becomes very long in the $t$-direction along the regions $V_\pm$.

By \Cref{lem:fredholm}, we can find a sequence of points $p_k \in W$ and almost-complex structures $\tilde{J}^k$ on $X_k$ agreeing with $J^k$ off of $W$ so that any $u_k
	\in \mathcal{M}^{\tilde{J}^k}(X_k,\Omega_k;p_k)$ must not be contained in the region $(-k-\ep,k+\ep)\times Y$.

	We now proceed to assume that the first alternative ECH capacity of $X$ is finite and derive a contradiction. In particular, say there is $L>0$ so that $c_1^\Alt(X,\Omega)<L$. Since they are all symplectomorphic, $c_1^\Alt(X_k,\Omega_k) < L$ for each $k$. From the definition of the alternative capacities, we conclude there is some closed $\tilde{J}^k$-holomorphic curve $u_k:\Sigma_k \to X_k$ passing through $p_k$ with energy
	\[\mathcal{E}(u_k) = \int_{\Sigma_k}u_k^\ast \Omega_k <L.\]
	By definition, this curve $u_k$ lives in $\mathcal{M}^{\tilde{J}^k}(X_k,\Omega_k;p_k)$. As such, it cannot be contained in the region $(-k-\ep,k+\ep)\times Y$. By passing to a subsequence of $k$, we can assume that all curves $u_k$ exit this region through the same hypersurface, either $\{-k-\ep\}\times Y$ or $\{k+\ep\}\times Y$ (or possibly both). Assume they all pass through $\{k+\ep\}\times Y$ and the following argument will proceed identically in the opposite case.

\begin{figure}[h]
		\includegraphics[scale=0.75]{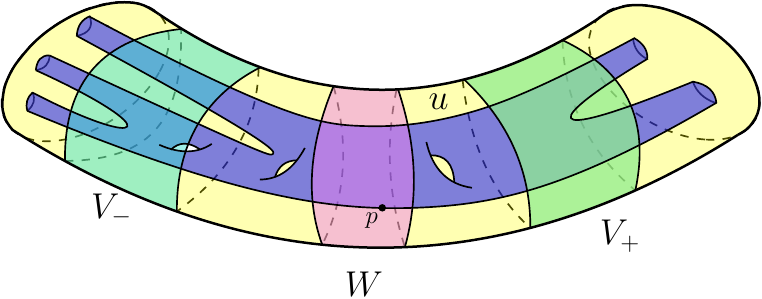}
		\includegraphics[scale=0.5]{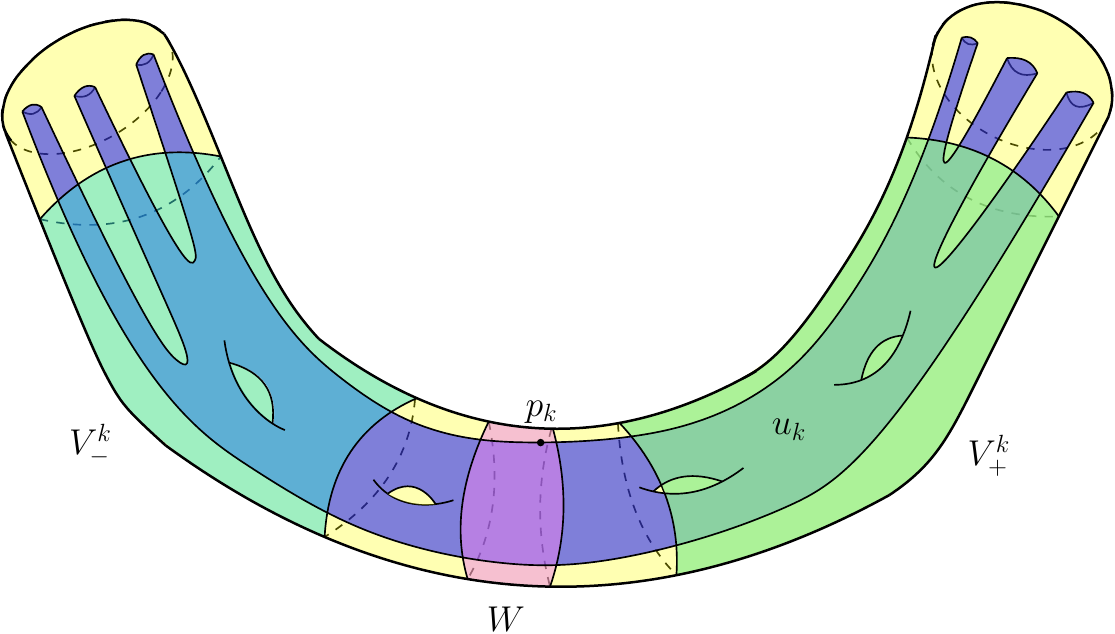}
		\includegraphics[scale=0.5]{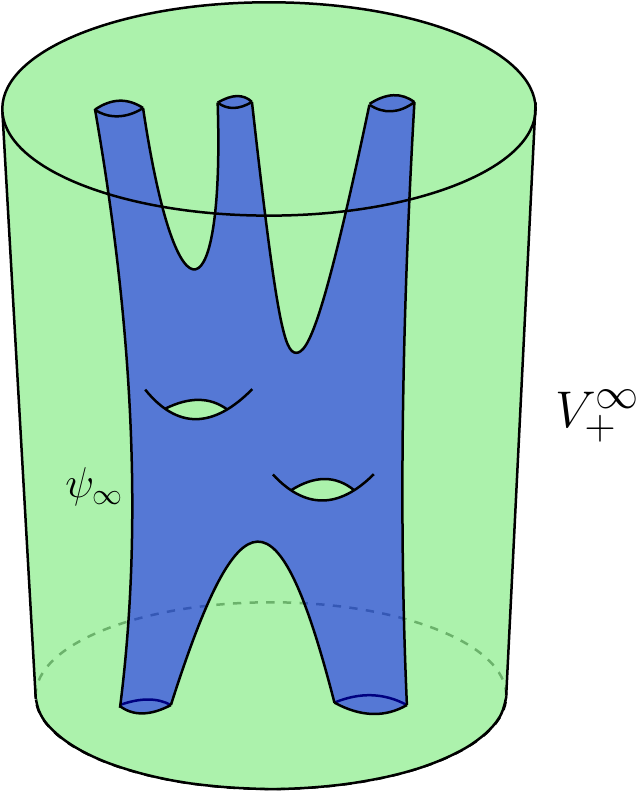}
		\begin{caption}{A schematic illustration of our neck-stretching procedure near the Zehnder hypersurface. The top picture is the original neighbourhood $U$ and a curve $u \in \mathcal{M}^{\tilde{J}}(X,\Omega;p)$. The bottom left is the stretched neighbourhood and a curve $u_k \in \mathcal{M}^{\tilde{J}^k}(X_k,\Omega_k;p_k)$. The bottom right is a curve $\psi_\infty$ in $V_\infty$ obtained as a limit of curves $u_k|_{V_+^k}$ in the stretched neck.}
		\label{fig:neck-stretch}
		\end{caption}
\end{figure}
	
	Perturbing $\ep$, we can restrict the domain and range of $u_k$ to define a sequence of somewhere-injective irreducible $J_+^k$-holomorphic maps
	\[\psi_k: C_k \to  \overline{V^k_+} =\left[\frac{\ep}{3},k+\frac{2\ep}{3}\right]\times Y,\]
	where $C_k$ is some compact Riemann surface with boundary mapping to the boundary of $\bar{V_+^k}$.
	Since the energy of a holomorphic curve is non-negative everywhere, we have
	\[\int_{C_k} \psi_k^\ast\big( \w + \dd(\phi_k(t)\lambda)\big) \leq \mathcal{E}(u_k) <L.\]
	
	By inclusion, we can view each curve $\psi_k(C_k)$ as living in the symplectic manifold
	\[(V_+^\infty,\Omega_\infty) = \Big(\left[\frac{\ep}{3},\infty\right)_t\times Y, \w +\dd(\phi_\infty(t)\lambda)\Big).  \]
	On compact sets, the symplectic forms $\Omega_k|_{V_+^k}$ will $C^\infty$-converge to $\Omega_\infty$. The almost-complex structures $J^k_+$ will $C^\infty$-converge on compact sets to an $(\lambda,\w)$-compatible $J^\infty_+$ on $V_+^\infty$. By definition, $J^\infty_+$ will also be $\Omega_\infty$-compatible.
	
	We have
	\begin{align*}
		\int_{C_k} \psi_k^\ast (\w+ \dd(\phi_\infty (t)\lambda)) &= \int_{C_k} \psi_k^\ast (\w+ \dd(\phi_k (t)\lambda))  + \int_{C_k} \psi_k^\ast \dd((\phi_\infty (t)-\phi_k(t))\lambda)\\
		& \leq \mathcal{E}(u_k) + \int_{ C_k}  \psi_k^\ast [(\phi'_\infty(t)-\phi'_k(t) )\dd{t} \wedge \lambda].
	\end{align*}
	By our choice of the sequence $\phi_k$, we have $\phi'_\infty(t) \leq \phi'_k(t)$ for all $t$. Moreover, $(\lambda,\w)$-compatibility of $\tilde{J}^k$ implies $\psi^\ast_k(\dd{t}\wedge\lambda)$ evaluates non-negatively pointwise on $C_k$. Hence,
	\begin{equation}\label{eq:cylenergybound}
		\int_{C_k} \psi_k^\ast \Omega_\infty < L.
	\end{equation}
	
	Because of the energy bound \eqref{eq:cylenergybound}, we can apply an exhaustive version of Gromov--Taubes compactness for holomorphic currents, as given in \cite[Prop.\!\ 3.8]{Tau}. This implies that there is a sequence $k_n \to \infty$ and a proper $J^\infty_+$-holomorphic map
	\[\psi_\infty: C_\infty \to V_+^\infty\]
	for some Riemann surface $C_\infty$ (possibly of infinite type) such that:
	
	\begin{enumerate}
		\item The $\psi_{k_n}$ converge weakly as currents to $\psi_\infty$. I.e., for any 2-form $\sigma$ with compact support on $V_+^\infty$,
		\[\lim_{n\to\infty} \int_{C_{k_n}} \psi_{k_n}^\ast \sigma =\int_{C_\infty} \psi_\infty^\ast \sigma. \]
		\item The curves converge on compact sets as point sets in the sense of Hausdorff. I.e., for any compact set $K\subset V_+^\infty$,
		\[\lim_{n\to\infty}\Big[ \sup_{x \in \psi_{k_n}(C_{k_n})\cap K}\mathrm{dist}(x, \psi_\infty(C_\infty)) + \sup_{y \in \psi_\infty (C_\infty)\cap K} \mathrm{dist}(y, \psi_{k_n}(C_{k_n}))\Big] = 0.\] 
	\end{enumerate}
	
	Convergence as currents plus \eqref{eq:cylenergybound} implies
	\[\int_{C_\infty} \psi_\infty^\ast \Omega_\infty< L.\]
	Recall that each $\psi_k$ met both boundary components of $V_+^k$. From point set convergence, it follows that $\psi_\infty(C_\infty)$ cannot  be contained in a compact subset of $V_+^\infty$. In particular, if $\pi : V_+^\infty \to  [\frac{\ep}{3}, \infty)$ is the projection to the $t$ coordinate, then $\pi \circ \psi_\infty$ is surjective.
	
	We claim that $\psi_\infty$ must be $C^0$-asymptotic to trivial cylinders over a collection of Hamiltonian orbits of the SHS $(\lambda,\w)$ as $t\to\infty$. The preceding paragraph implies this collection is non-empty. On the other hand, $(\lambda,\w)$ has no closed orbits, giving a contradiction and completing the proof.
	
	 To see our claim that $\psi_\infty$ asymptotes to cylinders over orbits, we follow some arguments of Hutchings \cite[\S 9] {Hutindineq}. For $N>0$, let $\tau_N: V_+^\infty \to V_+^\infty$ denote the forward translation in the $t$-direction by $N$. Define $C_\infty^N = \psi_\infty^{-1}(\im(\tau_N))$ and let 
	 \[\psi^N_\infty = \tau_{N}^{-1} \circ \psi_\infty|_{C^N_\infty}:C_\infty^N\to V_+^\infty,\]
	 which makes sense because $\tau_N$ is invertible on its image. 
	 We have 
	 \begin{align*}
	 	\int_{C_\infty^N} (\psi_\infty^N)^\ast \Omega_\infty  &= \int_{C_\infty^N} \psi^\ast_\infty(\w + \dd(\phi_\infty(t+N)\lambda))\\
	 	&\leq  \int_{C_\infty} \psi_\infty^\ast \Omega_\infty  + \int_{C_\infty} \psi_\infty^\ast[( \phi'_\infty(t+N)-\phi'_\infty(t)) \dd{t}\wedge\lambda].
	 	\intertext{Because $\psi_\infty^\ast (\dd{t}\wedge\lambda)$ is pointwise non-negative,} 
	 	&< L + \max_{t\geq \frac{\ep}{3}}| \phi'_\infty(t+N)-\phi'_\infty(t)| \int_{C_\infty} \psi_\infty ^\ast (\dd{t}\wedge \lambda)\\
	 	&\leq  L + \int_{C_\infty} \psi_\infty ^\ast (\dd{t}\wedge \lambda).
	 \end{align*}
	 Hence the sequence of curves $\psi_\infty^N$ has bounded energy. Again by Gromov--Taubes compactness, we can pass to a subsequence $(N_j)_{j\in\N}$ so that they converge as a current and a point set to some curve $\eta : C \to V_+^\infty$. We claim that $\eta$ is a collection of trivial cylinders over orbits; convergence as a point set then implies $\psi_\infty$ must be asymptotic to cylinders over orbits.
	 
	 Given any compact set $K \subset V_+^\infty$, we may consider the translates $\tau_{N_k} K$ which are all disjoint on some subsequence $(N_k)_{k\in\N}\subset (N_j)_{j\in\N}$. Then one has 
	 \[L>\int_{C_\infty}\psi_\infty^\ast\w \geq \sum_{k=1}^\infty \int_{C_\infty \cap \psi_\infty^{-1}(\tau_{N_k} K)}\psi_\infty^\ast\w = \sum_{k=1}^\infty \int_{ \psi_\infty^{N_k}(C_\infty^{N_k})\cap K}\w. \]
	Along with current convergence, this implies
\[0 = \lim_{k\to\infty} \int_{\psi_\infty^{N_k}(C_\infty^{N_k})\cap K}\w= \int_{\eta(C)\cap K} \w= \int_{C\cap \eta^{-1}(K)}\eta^\ast \w.\]
	 Since this holds for any compact $K$, it must be that
	 \[\int_{C} \eta^\ast \w=0.\]
	 Now we conclude as in \cite[Prop.\!\ 9.1]{Hutindineq}; for any tangent vector $v \in TV_+^\infty$, $\w(v,Jv)\geq 0$ with equality if and only if $v$ is in the span of $\pd_t$ and $R$. Thus, the image of $\eta$ must be of the form $\R_t \times \Lambda$, for some compact invariant set $\Lambda$ of the flow. Because the energy $\int_C \eta^\ast \Omega_\infty$ is finite, $\Lambda$ must be a collection of closed orbits. 
\end{proof}
	
\begin{remark}
	In \cite[\S 5.5]{Ush}, Usher also constructs symplectic 4-manifolds that contain certain hypersurfaces which are non-trivial torus bundles over the circle (for example by Gompf sum along symplectic tori with non-trivial normal bundles). These Hamiltonian hypersurfaces are no longer stable, but still fail the nearby existence property. We expect our results to carry over in the presence of these hypersurfaces after replacing SFT neck-stretching with a suitable version of the ``feral curve" adiabatic neck-stretching techniques from \cite{FH, Pra}, but there are some complications which we do not claim to address here.
\end{remark}

\begin{remark}
	It is much easier to show ``elementary capacities" are infinite in higher dimensions due to the mechanics of the Fredholm index. Consider the capacities $c_{k,\infty}(X,\w)$ which are introduced by Hutchings \cite{Hut26} in any dimension. To define these, one first defines for $g\geq 0$ the capacities $c_{k,g}(X,\w)$ which have the same definition as the alternative ECH capacities but require that the holomorphic curves measured by the capacities have total genus at most $g$. Then we take $c_{k,\infty}(X,\w)$  as the limit of the decreasing sequence $c_{k,g}(X,\w)$ as $g \to \infty$.  
	
	Let $M$ be a closed oriented manifold of dimension six or greater admitting a metric $g$ with negative sectional curvature, and let $D^\ast_g M$ be the associated unit cotangent disk bundle. As observed by Eliashberg--Givental--Hofer \cite[Cor.\!\! 1.7.4]{EGH}, the Fredholm index of any holomorphic curve in $(D^\ast_g M,\w_\can)$ is at most zero. Hence for generic $J$, there cannot be a somewhere injective curve through a generic point of $D^\ast_g M$. One concludes $c_{1,\infty}(D^\ast_g M,\w_\can)$ is infinite. The same considerations show $c_{1,\infty}(X,\w)$ is infinite for any closed 6-manifold which is symplectic Calabi--Yau (i.e. $c_1(TX,\w)=0$) or any higher-dimensional closed symplectic manifold which is both symplectic Calabi--Yau and symplectically aspherical (i.e. $\w(\pi_2(M))=0$).
\end{remark}

\subsection{Elementary tame capacities and K\"ahler surfaces}
One can define a version of the alternative ECH capacities using almost the same construction, but allowing the almost-complex structures to be merely $\w$-tame rather than $\w$-compatible. 

Let $(X,\w)$ be an admissible symplectic 4-manifold in the same sense as the definition of the alternative capacities. Let $\mathcal{J}^\mathrm{tame}(X,\w)$ denote the space of almost-complex structures on the symplectic completion $(\hat{X},\hat{\w})$ which are $\hat{\w}$-tame everywhere and $\lambda$-compatible on any symplectization ends. Now we define the \emph{elementary tame capacities}
\[0= c_0^\mathrm{tame}(X,\w) < c_1^\mathrm{tame}(X,\w) \leq c_2^\mathrm{tame}(X,\w) \leq \cdots \leq \infty\]
by setting
\begin{equation}\label{eq:tamedef}
	c_k^{\mathrm{tame}}(X,\w) = \sup_{\substack{{J\in \mathcal{J}^\mathrm{tame}(X,\w)}\\ {x_1,\ldots, x_k \in X \,\, \mathrm{distinct}}}}  \inf_{u\in \mathcal{M}^J(X,\w;\, x_1,\ldots, x_k)} \mathcal{E}(u).
\end{equation}
As usual, we extend this to an arbitrary symplectic 4-manifold $(X,\w)$ by declaring $c_k^\mathrm{tame}(X,\w)$ to be the supremum of $c_k^\mathrm{tame}$ of any admissible symplectic 4-manifold which embeds into $(X,\w)$.

Since $ \mathcal{J}(X,\w)\subset \mathcal{J}^\mathrm{tame}(X,\w)$, it is immediate that
\[c_k^\Alt(X,\w) \leq c_k^{\mathrm{tame}}(X,\w).\]
The tame capacities satisfy all the properties of the alternative ECH capacities from \Cref{prop:altproperties}, except possibly the ECH bound. The proofs of most of these properties are clear, following the arguments of \cite{HutAlt}. The ECH index property holds after noting that Hutchings' proof of the ECH index inequality $\ind(u) \leq I([u])$ is valid for $\w$-tame almost-complex structures (see \cite{Hutind}). The Seiberg--Witten bound uses that Gromov--Taubes invariants are defined in the generality of tame almost-complex structures 

Note that ECH has not been developed for tame almost-complex structures, which is why it is unknown (although likely) that one has an inequality $c_k^\tame(X,\w)\leq c_k^\ECH(X,\w)$ for a Liouville domain $(X,\w)$. 

It follows from the proofs in \cite[Thm.\!\! 9 and Thm.\!\! 15]{HutAlt}, which rely only on the properties of \Cref{prop:altproperties}, that
\[c_k^\tame(X,\w) = c_k^\Alt(X,\w) = c_k^\ECH(X,\w)\]
for $(X,\w)$ any convex or concave toric domain. For the same reason, the tame capacities will also satisfy the conclusions of \Cref{thm:alt_fin}.

Our motivation for introducing these capacities is that one can show they are infinite in relatively large generality using some work of Lee--Parker \cite{LP}, and one might venture that the same is true for the alternative ECH capacities.

\begin{proposition}\label{prop:tame}
	Let $(X,\w)$ be a closed K\"ahler surface with $b_2^+(X)>1$. Then the capacities $c_k^\tame(X,\w)$ are infinite.
\end{proposition}
\begin{proof}
	Let $(X,\w)$ be a K\"ahler surface with associated integrable complex structure $J$ and K\"ahler metric $g$. Suppose $\alpha \in H^{2,0}(X)$ is a holomorphic 2-form, which we can identify with an element of $\Re(H^{2,0}\oplus H^{0,2})$. By the Hodge index theorem, such a non-zero $\alpha$ exists if and only if $b_2^+(X)>1$. 
	
	Following Lee--Parker \cite{LP}, we define a map $K_\alpha: TX\to TX$ by
	\[g(u, K_\alpha v) = \alpha(u,v).\]
	We then introduce an almost-complex structure $J_\alpha$ given as
	\begin{equation}\label{eq:Jalpha}
		J_\alpha = (I+JK_\alpha)^{-1} J (I+JK_\alpha).
	\end{equation}
	By some algebra performed in \cite[Prop.\!\! 1.5]{Lee}, we have that
	\[g(J_\alpha u,J_\alpha v) = g(u,v) \qand J_\alpha = \frac{1-|\alpha|^2}{1+|\alpha|^2} J - \frac{2}{1+|\alpha|^2}K_\alpha.\]
	Then 
	\begin{align*}
		\w(u, J_\alpha u)&= g(Ju, J_\alpha u) \\
		& = \frac{1-|\alpha|^2}{1+|\alpha|^2}g(Ju, Ju) -\frac{2}{1+|\alpha|^2}g(Ju, K_\alpha u)\\
		&=  \frac{1-|\alpha|^2}{1+|\alpha|^2}g(u,u) +\frac{2}{1+|\alpha|^2}\alpha(u,Ju).
	\end{align*}
	The latter term always vanishes because $\alpha(u,Ju)=g(u,JK_\alpha u)$ and $JK_\alpha$ is skew-adjoint. We obtain 
	\[\w(u, J_\alpha u)=  \frac{1-|\alpha|^2}{1+|\alpha|^2} \|u\|^2_g.\] 
	One concludes that $J_\alpha$ is $\w$-tame if and only if $|\alpha|<1$ everywhere; this can be achieved by rescaling $\alpha$.  It is not needed for the subsequent argument, but we note 
	\begin{align*}
		\w(J_\alpha u, J_\alpha v) &=  \frac{1-|\alpha|^2}{1+|\alpha|^2}g(J_\alpha u,v) -\frac{2}{1+|\alpha|^2}\alpha(JJ_\alpha u,  v)\\
		&= \qty(\frac{1-|\alpha|^2}{1+|\alpha|^2})^2g(J u,v)-\frac{2-2|\alpha|^2}{(1+|\alpha|^2)^2}\Big(g(K_\alpha u,v)-\alpha(u,v)\Big)\\&\hspace{2cm}  +\qty(\frac{2}{1+|\alpha|^2})^2\alpha(JK_\alpha u,  v)\\
		&= \left( 1-\frac{8|\alpha|^2}{(1+|\alpha|^2)^2} \right)  \w(u,v)+\frac{4-4|\alpha|^2}{(1+|\alpha|^2)^2}\alpha(u,v) .
	\end{align*}
	In particular, $J_\alpha$ is $\w$-compatible if and only if $\alpha=0$, in which case $J_\alpha=J$.
	
	If $\mathrm{PD}(A)\in H^{1,1}(X)$ is non-zero, then Lee--Parker \cite[Lem. 3.2]{LP} show that any connected $J_\alpha$-holomorphic curve representing the class $A$ must be contained in the divisor $D_\alpha$ given as the zero set of the holomorphic form associated to $\alpha$. In fact their proof implies something stronger. Namely, if $u: \Sigma \to X$ is a connected non-constant $J_\alpha$-holomorphic curve so that
	\begin{equation}\label{eq:LPcond}
		\int_\Sigma u^\ast \alpha =0, \qq{then} u(\Sigma) \subset D_\alpha.
	\end{equation}
	
	Now suppose for the sake of eventual contradiction that $c_1^\tame(X,\w)= L<\infty$. Pick a non-zero holomorphic 2-form $\alpha$; by deforming and rescaling, we may suppose $|\alpha|^2<1$ everywhere and that $[\alpha]$ defines a rational class in the cohomology of $X$. Pick a point $x \in X$ which is not in the divisor $D_\alpha$. Consider the sequence of 2-forms $\alpha_k = \frac{1}{k} \alpha$ for $k\in\N$ and the associated almost-complex structures $J_k=J_{\alpha_k}$. Since each $J_k$ is $\w$-tame, by definition \eqref{eq:tamedef}, 
	\[L=c_1^\tame(X,\w) \geq \inf \{\mathcal{E}(u) : u \in \mathcal{M}^{J_{k}}(X,\w; x)\} \qq{for} k\in\N.\]
	In particular, we can find a sequence of $J_{k}$-holomorphic curves $u_k: \Sigma_k \to X$ which each pass through $x$ and have uniformly bounded energy:
	\begin{equation}\label{eq:energybound}
		\int_{\Sigma_k} u_k^\ast\w \leq L.
	\end{equation}
	As in \cite[Rem.\!\! 2]{HutAlt}, we may as well assume that $u_k$ is connected, irreducible, and somewhere-injective. Let $A_k \in H_2(X;\Z)$ denote the homology class $[u_k]$ of the $k$th curve. Since $x$ does not lie in the divisor $D_\alpha$, we conclude from \eqref{eq:LPcond} that 
	\[\alpha(A_k) = k\int_{\Sigma_k} u_k^\ast \alpha_k \neq 0.\]
	In fact, the proof of \cite[Lem. 3.2]{LP} also tells us 
	\[\alpha(A_k) = \int_{\Sigma_k} |\overline{\pd}_Ju|^2,\]
	hence $\alpha(A_k)>0$ for all $k$. Note $\alpha$ and $A_k$ belong to rational cohomology and integral homology respectively. Let $N$ be the least common denominator of $\alpha$, so that $N\alpha$ lifts to $H^2(X;\Z)$. Then $N\alpha(A_k)$ is a positive integer, in particular at least one. We conclude that $\alpha(A_k) >N^{-1}$ for all $k$.
	
Note that $K_{\alpha_k} = \frac{1}{k} K_\alpha$, and so the sequence $K_{\alpha_k}$ will $C^\infty$-converge to the zero endomorphism. And hence by \eqref{eq:Jalpha}, the sequence of almost-complex structures $J_k$ will $C^\infty$-converge to the integrable complex structure $J$. Since we have a uniform energy bound \eqref{eq:energybound}, Gromov--Taubes compactness applies to say the sequence $u_k$ has a subsequence which converges weakly as a current to a $J$-holomorphic current $\mathcal{C} \subset  X$. 

 Current convergence implies we can pass to a subsequence so that $[u_k]=[\mathcal{C}]$ for all $k$. The adjunction formula then gives a uniform upper bound on the genus of $u_k$, and, since each $u_k$ is connected, we conclude $\mathcal{C}$ is the image of a connected finite genus $J$-holomorphic curve.

Since $J$ is integrable, $\mathcal{C}$ is essentially an effective divisor in the K\"ahler surface $(X,\w)$ and so it must be Poincar\'e dual to a $(1,1)$-class and pair trivially with $\alpha$. On the other hand, $\alpha([u_k]) > N^{-1}$ for each $k$ and so $\alpha([\mathcal{C}])>0$. We thus have a contradiction. 	
\end{proof}
\begin{remark}
	The above argument does not apply to the capacities $c_k^{\Alt}$ simply because, as demonstrated in the course of the proof, these almost-complex structures $J_\alpha$ are never $\w$-compatible when $\alpha\neq 0$.
\end{remark}

\subsection{Gromov--Taubes invariants and finite alternative capacities}
In this section we prove \Cref{thm:alt_fin}. We briefly recall the Gromov--Taubes invariant, which will be used in the proof. 

Let $(X,\w)$ be a closed symplectic 4-manifold. Taubes \cite{TauSWGr} introduced a function
\[\mathrm{Gr}(X,\w;\cdot,\cdot) : H_2(X;\Z)\otimes\Lambda^{\mathrm{even}} (H_1(X;\Z)/\mathrm{torsion}) \to \Z.\]
Given $A\in H_2(X;\Z)$ and a collection of classes $\gamma_1,\ldots, \gamma_{2k} \in H_1(X;\Z)/\mathrm{torsion}$, the Gromov--Taubes invariant $\mathrm{Gr}(X,\w;A, \gamma_1,\ldots, \gamma_{2k})$ is defined as follows. We pick some generic $\w$-tame almost-complex structure $J$ and a generic collection of 1-cycles $\Gamma$ representing the classes $\gamma_1,\ldots, \gamma_{2k}$. Setting
\[I(A) = A^2 +\langle c_1(TX,\w),A\rangle ,\]
suppose $I(A)\geq 2k$. Then we pick a set $\mathbf{p}$ of $\frac{1}{2}I(A)-k$ generic points. The Gromov--Taubes invariant is equal to a weighted sum over collections of embedded* $J$-holomorphic curves which pass through the cycles $\Gamma$, the  points $\mathbf{p}$, and have total homology class $A$. We also need to include multiply covered $J$-holomorphic tori with zero self-intersection and weight them in a fairly subtle way (see e.g.\!\! \cite{McD}). If $I(A) <2k$, we define the Gromov--Taubes invariant to be zero. 

In parallel, for a closed 4-manifold $X$ with $b_2^+\geq 2$, there are the Seiberg--Witten invariants
\[SW(X, \cdot, \cdot): \mathrm{Spin}{}^c(X)\otimes \Lambda^\ast (H_1(X;\Z)/\mathrm{torsion}) \to \Z,\]
defined by integrating suitable cohomology classes over the moduli space of solutions to the Seiberg--Witten equations. When $b_2^+=1$, there are two such invariants $SW_\pm$ distinguished by a choice of chamber for the perturbation term in the Seiberg--Witten equations. Recall that $\mathrm{Spin}{}^c(X)$ is an affine space modelled on $H^2(X)$ and a symplectic form $\w$ determines a canonical spin$^{c}$-structure $\mathfrak{s}_\w$. For a symplectic 4-manifold $(X,\w)$ with $b_2^+>1$, Taubes \cite{TauSWGr} famously proves an identification
\[\mathrm{Gr}(X,\w; A, \gamma_1,\ldots, \gamma_{2k}) = SW(X,\mathfrak{s}_\w +\mathrm{PD}(A), \gamma_1,\ldots, \gamma_{2k}).\]
When $b_2^+(X)=1$, the symplectic form $\w$ specifies a chamber for the Seiberg--Witten invariants\footnote{The opposite orientation convention is taken in some places, e.g. \cite{LL}} and Taubes' identification extends to give
\[\mathrm{Gr}(X,\w; A, \gamma_1,\ldots, \gamma_{2k}) = SW_+(X,\mathfrak{s}_\w +\mathrm{PD}(A), \gamma_1,\ldots, \gamma_{2k})\]
provided the pairing $A\cdot S \geq -1$ for any embedded symplectic sphere $S$ with $S^2=-1$. When $A$ fails this condition, McDuff \cite{McD} constructs a modified invariant $\mathrm{Gr}'$ and Li--Liu \cite{LL} prove this agrees with $SW_+$ in the remaining cases.

The following slight enhancement of \cite[Lem.\!\ 18]{HutAlt} suffices for our purposes.
\begin{lemma}\label{lem:GrAlt}
	Suppose $\mathrm{Gr}(X,\w;A, \gamma_1,\ldots, \gamma_{2d})\neq 0$ or $\mathrm{Gr}'(X,\w;A, \gamma_1,\ldots, \gamma_{2d})\neq 0$  with  $2k\leq I(A)-2d$. Then
	\[c_k^\Alt(X,\w) \leq \w(A). \]
\end{lemma} 
\begin{proof}
	By definition of the Gromov--Taubes invariants or McDuff's modification, non-vanishing of the appropriate invariant implies for any generic $\w$-compatible $J$ and generic collection of $k$ points, there is a possibly disconnected $J$-holomorphic curve $u$ through those points representing the homology class $A$. The energy of this curve is then $\mathcal{E}(u)=\w(A)$. By the adjunction formula, there is an upper bound on the genus of this curve. By Gromov compactness, we can then find a curve of energy at most $\w(A)$ for arbitrary $J$ and an arbitrary collection of $k$ points by taking a sequence of generic $J$s and generic points converging to them and studying the associated sequence of curves.
\end{proof}

\begin{proof}[Proof of \Cref{thm:alt_fin}]
By Buse--Hind--Opshtein's proof of packing stability \cite{BHO}, for a sufficiently small $\lambda$ and some $n\in\N$, there is a symplectic embedding
\[\coprod_{i=1}^n \mathrm{int}(B^4(\lambda)) \hookrightarrow (X,\w) \qq{with} \vol(X,\w) = n\cdot \vol(B^4(\lambda)).\]
Since any collection of $n$ closed balls with radii less than $\lambda$ embed into $(X,\w)$, continuity and monotonicity of the alternative capacities implies
\[c_k^\Alt(X,\w) \geq c_k^\Alt\left(\coprod_{i=1}^n B^4(\lambda) \right).\]
Because the alternative capacities of a union of balls agree with its ECH capacities, the ECH Weyl law yields
\[\lim_{k\to \infty} \frac{c_k^\Alt(X,\w)^2}{4k} \geq \lim_{k\to \infty} \frac{c_k^\ECH(\coprod_{i=1}^n B^4(\lambda))^2}{4k} =n\cdot \vol(B^4(\lambda))= \vol(X,\w).\]
If we consider the the error terms $e_k^\Alt(X,\w) := c_k^\Alt(X,\w) - \sqrt{4\vol(X,\w)k}$, monotonicity and equality of volumes implies
\[e_k^\Alt(X,\w) \geq e_k^\Alt\left(\coprod_{i=1}^n B^4(\lambda)\right)= e_k^\ECH\left(\coprod_{i=1}^n B^4(\lambda)\right).\]
The error terms $e_k^\ECH=c_k^\ECH-\sqrt{4k\vol}$ for a union of balls are uniformly bounded in $k$, as shown in \cite[\S10]{EdtPack}. Thus the error terms $e_k^\Alt(X,\w)$ have a uniform lower bound independent of $k$.

It remains to find a $k$-independent upper bound for the error terms $e_k^\Alt$. We will do so by bounding the alternative capacities of $(X,\w)$ in terms of some Gromov--Taubes invariants, following Usher \cite{Ush} (see also the work of Prasad \cite[Prop.\! A.5]{Pra} and Li--Ning \cite[Prop.\!\ 4.16]{LNalg} for some similar computations). Assume for the moment that $(X,\w)$ is rational. By rescaling, we may assume that $(X,\w)$ is integral, and moreover that 
\begin{equation}\label{eq:posdimcond}
[\w]\cdot[\w] \geq 4+ b_1(X)- [\w]\cdot c_1(TX,\w).	
\end{equation}
Pick some basis $\gamma_1,\ldots , \gamma_{b_1}$ for $H_1(X;\Z)/\mathrm{torsion}$.
As Usher \cite[Appendix]{Ush} computes, under the assumption \eqref{eq:posdimcond},  the wall-crossing and charge conjugation properties of the Seiberg--Witten invariants imply 
\[SW_+(X,\mathrm{PD}[\w], \gamma_1,\ldots,\gamma_{b_1})=\pm 1.\]
By an identical argument, if $n$ is any positive integer, then
\[SW_+(X,\mathrm{PD}[n\w], \gamma_1,\ldots,\gamma_{b_1})=\pm 1.\]
Certainly the class $\mathrm{PD}[n\w]$ pairs positively with any $\w$-symplectic sphere. Hence Taubes' equivalence $SW=\mathrm{Gr}$ applies and we obtain
\begin{equation}\label{eq:Grnonzero}
	\mathrm{Gr}(X,\w; \mathrm{PD}[n\w], \gamma_1,\ldots,\gamma_{b_1})=\pm 1.
\end{equation}
Set
\[k_n = \frac{1}{2}I(\mathrm{PD}[n\w])-\frac{1}{2}b_1(X)=\frac{n^2}{2}[\w]\cdot[\w] +\frac{n}{2}[\w]\cdot c_1(TX,\w) -\frac{1}{2}b_1(X).\]
Let $\kappa = [\w]\cdot c_1(TX,\w) $. Note $2\vol(X,\w) =  [\w]\cdot[\w]$. Also abbreviate $b_1=b_1(X)$ and $\vol=\vol(X,\w)$. We may then rewrite
\[k_n= n^2 \vol +\frac{n}{2}\kappa -\frac{1}{2}b_1.\]
We may take $n$ sufficiently large so that $k_{n-1}<k_n$. From \Cref{lem:GrAlt} and \eqref{eq:Grnonzero}, we have that for $k_{n-1} < k \leq k_n$,
\[c_{k}^\Alt(X,\w) \leq \langle [\w], \mathrm{PD}[n\w]\rangle =2n\vol(X,\w).\]
Then, letting $n$ denote the smallest integer for which $k\leq k_n$,
\begin{align*}
	e_{k}^\Alt(X,\w) &\leq 2n\vol -\sqrt{4k\vol }\\
	&= \frac{4n^2 \vol^2 -4k \vol}{2n\vol +\sqrt{4k\vol }}\\
	&<  \frac{4n^2 \vol^2 -4k_{n-1}\vol}{2n\vol +\sqrt{4k\vol }}\\
	&= \frac{8n\vol^2- 4\vol^2 -2n\kappa\vol +2\kappa \vol +2b_1\vol }{2n\vol +\sqrt{4k\vol}} \\
	&\leq \max\left\{0,4\vol -\kappa +\frac{-2\vol  + \kappa +b_1}{n}\right\}\\
	&\leq \max\{0,b_1+4\vol  +|\kappa|\}.
\end{align*}
Hence the error term is bounded from above.

Now we wish to consider an arbitrary non-rational symplectic form $\w$. Consider an integral basis $\alpha_1,\ldots, \alpha_{b_2}$ for $H^2(X;\Z)/\mathrm{torsion}$ with respect to which the intersection form is either the matrix $\mathrm{diag}(1, -1,\cdots ,-1)$ or the matrix $H =\pmqty{0& 1\\1 &0}$. Such a basis exists by the classification of indefinite intersection forms and the fact $b_2^+=1$. Using this basis, we can define a norm on $H^2(X;\R)$:
\[\|v_1 \alpha_1 +\cdots + v_{b_2} \alpha_{b_2}\| = \sqrt{v_1^2 + \cdots + v_{b_2}^2}. \]
Note by Cauchy--Schwarz
\[| \alpha \cdot \beta|  \leq \|\alpha\|\cdot \|\beta\|.\]

In our integral basis, write 
\[[\w] = a_1\alpha_1+\cdots + a_{b_2}\alpha_{b_2}.\] 
for $a_i\in \R$. We can find a sequence of rational cohomology classes $A_n$ which are close to $[\w]$. For any positive integer $n$, there are integers $d_1,\ldots, d_n$ defining an integral cohomology class $B_n=d_1\alpha_1+\cdots + d_{b_2}\alpha_{b_2}$ so that
\[\text{for $1\leq i\leq b_2$,} \quad\quad |na_i - d_i| \leq \frac{1}{2}.  \]
Then $A_n = n^{-1}B_n$ are the classes nearby to $[\w]$. In particular,
\[\|B_n - n[\w]\| \leq \frac{1}{2}\sqrt{b_2} \qand \|A_n - [\w]\| \leq \frac{\sqrt{b_2}}{2n}.\]

Recall the space of cohomology classes realized by a symplectic form on $X$ is an open cone. In particular, we can restrict to $n$ large enough so that each cohomology class $A_n$ is represented by a rational symplectic form $\w_n$. Further, without loss of generality, by rescaling, we can suppose all the classes $B_n$ satisfy \eqref{eq:posdimcond}. Note here we use that the forms $\w_n$ are all homotopic and thus $c_1(TX,\w_n)$ is independent of $n$ and agrees with $c_1(TX,\w)$. Let $\Omega_n=n\w_n$ be the associated integral symplectic form representing the class $B_n$. We compute
\begin{equation}\label{eq:Cconst}
\begin{split}
\|[\Omega_{n+1}]-[\Omega_n]\| &\leq  \|B_{n+1}-(n+1)[\w]\|+\|B_n -n[\w]\| + \|[\w]\|\\
& \leq \sqrt{b_2} +\|[\w]\|
\end{split}
\end{equation}
and
\begin{align*}
	([\w]\cdot [\Omega_n])^2 &= \big([\w]\cdot (n[\w] + [\Omega_n]-n[\w])\big)\big(([\w]-n^{-1}[\Omega_n]+n^{-1}[\Omega_n])\cdot  [\Omega_n]\big)\\
	&= \big(n[\w]\cdot[\w]+ [\w]\cdot ([\Omega_n]-n[\w])\big)\big(n^{-1}[\Omega_n]\cdot[\Omega_n]+  (n[\w]-[\Omega_n])\cdot[\omega_n]\big).
\end{align*}
Hence,
\begin{align*}
	\big|([\w]\cdot [\Omega_n])^2 - ([\w]\cdot[\w])([\Omega_n]\cdot[\Omega_n])\big| &\leq n^{-1}|[\w]\cdot([\Omega_n]-n[\w])|\cdot |[\Omega_n]\cdot[\Omega_n]|\\
	&\hspace{1cm}+ n|[\w]\cdot[\w]|\cdot |(n[\w]-[\Omega_n])\cdot[\w_n]|\\
	&\hspace{1cm} + |[\w]\cdot([\Omega_n]-n[\w])|\cdot|(n[\w]-[\Omega_n])\cdot[\w_n]|\\
	&\leq \frac{n\sqrt{b_2}}{2}\big(\|[\w_n]\|^2\cdot \|[\w]\| + \|[\w_n]\|\cdot \|[\w]\|^2\big)\\
	&\hspace{1cm} + \frac{b_2}{4}\|[\w]\|\cdot \|[\w_n]\| \\
	&\leq Dn,
\end{align*}
where $D$ is some constant independent of $n$.

For $n$ large, $\w$ and $\w_n$ will determine the same chamber for the Seiberg--Witten invariants, and so as before we have
\[SW_+(X,\mathrm{PD}[\Omega_n],\gamma_1,\ldots,\gamma_{b_1})=\pm 1.\]
 The class $\Omega_n$ may not pair positively with all exceptional $\w$-symplectic spheres, nevertheless, we can apply Li--Liu's result \cite{LL} to obtain, for large $n$,
\[\mathrm{Gr}'(X,\w ; \mathrm{PD}[\Omega_n],\gamma_1,\ldots,\gamma_{b_1})=\pm 1. \]
If we write
\[j_n = \frac{1}{2}  [\Omega_n]\cdot [\Omega_n] + \frac{1}{2}  [\Omega_n]\cdot c_1(TX,\w)  -\frac{1}{2}b_1,\]
then for $j\leq j_n$ and $n$ sufficiently large, \Cref{lem:GrAlt} yields
\[c_j^\Alt(X,\w) \leq  [\w]\cdot [\Omega_n].\]
For $n$ sufficiently large, we have $j_{n}\leq j_{n+1}$. For $j$ sufficiently large, let $n$ denote the smallest integer for which $j \leq  j_{n+1}$. Then we compute
\begin{align*}
	e_j^\Alt(X,\w)&\leq  [\w]\cdot [\Omega_{n+1}]- \sqrt{4j\vol(X,\w)}\\
	&= \frac{( [\w]\cdot [\Omega_{n+1}])^2 - 4j\vol(X,\w)}{ [\w]\cdot [\Omega_{n+1}] + \sqrt{4j\vol(X,\w)}}.
	\intertext{Let $C= \|\w\|^2 + \sqrt{b_2}\|\w\|$. From \eqref{eq:Cconst} and the fact $j \geq  j_{n}$,}
	&\leq  \frac{([\w]\cdot [\Omega_{n}] + C)^2 - 4j_n\vol(X,\w)}{[\w]\cdot [\Omega_n] + \sqrt{4j\vol(X,\w)}}\\
	&=  \frac{([\w]\cdot [\Omega_n])^2 +2C[\w]\cdot[\Omega_n] + C^2}{[\w]\cdot [\Omega_n] + \sqrt{4j\vol(X,\w)}} - [\w]\cdot[\w]\frac{[\Omega_n]\cdot [\Omega_n] + [\Omega_n]\cdot c_1(TX,\w) -b_1}{[\w]\cdot [\Omega_n] + \sqrt{4j\vol(X,\w)}}\\
	&\leq \frac{Dn+2C[\w]\cdot[\Omega_n] + C^2 +[\w]\cdot[\w]\Big(b_1- [\Omega_n]\cdot c_1(TX,\w)\Big)}{[\w]\cdot [\Omega_n] + \sqrt{4j\vol(X,\w)}}\\
	&\leq \frac{1}{n[\w]\cdot[\w_n]}\max\left\{0, Dn+2Cn[\w]\cdot[\omega_n] + C^2+[\w]\cdot[\w]\left(b_1-[\Omega_n]\cdot c_1(TX)\right)\right\}\\
	&= \max\left\{0, 2C + \frac{D - ([\w]\cdot[\w]) ( [\w_n]\cdot c_1(TX) )}{[\w]\cdot [\w_n]} + \frac{C^2+b_1 [\w]\cdot[\w]}{n[\w]\cdot [\w_n]}\right\}.
\end{align*}
Since $\|[\w]-[\w_n]\| \leq \frac{\sqrt{b_2}}{2n}$, the final expression can be bounded independent of $n$. We conclude $e_j^\Alt(X,\w)$ has a $j$-independent upper bound.
\end{proof}

\begin{proof}[Proof of \Cref{cor:alt_fin}]
	Suppose $(X,\w)$ is a compact symplectic 4-manifold with smooth boundary which has a symplectic embedding $\varphi: (X,\w)\hookrightarrow (W,\Omega)$ into a closed symplectic 4-manifold with $b_2^+(W)=1$. 
	
	First, by Edtmair's recent breakthrough in packing stability \cite{EdtPack}, provided $X$ has smooth boundary, one can find a collection of small enough balls $\coprod_{i=1}^n B^4(a)$ with total volume equal to $\vol(X,\w)$ so that $\coprod_{i=1}^n B^4(b)$ symplectically embeds into $(X,\w)$ for any $b<a$. Using the same argument as at the beginning of the proof of \Cref{thm:alt_fin}, or in \cite[\S10]{EdtPack}, it follows that
	\[\lim_{k\to \infty} \frac{c_k^\Alt(X,\w)^2}{4k} \geq \vol(X,\w)\]
	and that the error terms $e_k^\Alt(X,\w)$ have a $k$-independent lower bound. 
	
	Now consider the symplectic 4-manifold $W \setminus \varphi(X)$. This also has smooth boundary, and again Edtmair's packing stability result applies. Thus one can find a collection of balls $\coprod_{i=1}^n B^4(a)$ with volume equal to $(W \setminus \varphi(X), \Omega)$ so that $\coprod_{i=1}^n B^4(b)$ symplectically embeds into $(W \setminus \varphi(X), \Omega)$ for any $b<a$. By continuity, monotonicity, and the disjoint union property of the alternative capacities, we have
	\[c_k^\Alt(X,\w ) \leq \inf_{\ell\geq 0} \left[c_{k+\ell}^\Alt(W,\Omega) - c_{\ell}^\Alt\left(\coprod_{i=1}^n B^4(a)\right)\right] . \]
	Since $(W,\Omega)$ and the union of balls satisfy a Weyl law (the former by \Cref{thm:alt_fin}), this yields
	\[\lim_{k\to \infty} \frac{c_k^\Alt(X,\w)^2}{4k} \leq \vol(W,\Omega)-n\vol(B^4(a))=\vol(X,\w).\]
	Additionally, because the Weyl law for $(W,\Omega)$ and the union of balls have bounded error term (again from \Cref{thm:alt_fin}), we have 
	\[e_k^\Alt(X,\w) \leq \inf_{\ell\geq 0} \left[\sqrt{4(k+\ell) \vol(W,\Omega)}-\sqrt{4\ell n\vol(B^4(a))}-\sqrt{4k \vol(X,\w)} \right ]+O(1).\]
	The right hand side admits a $k$-independent upper bound, as in \cite[\S10]{EdtPack}. 
\end{proof}

\subsubsection{A question concerning equidistribution}

One notes that the preceding proof of \Cref{thm:alt_fin} suggests that when $k$ is very large, the holomorphic curves realizing the $k$th alternative capacity represent a homology class very close to a multiple of $\mathrm{PD}[\w]$. A similar observation and pertinent question appear in the work of Li--Ning \cite[p.\! 8403]{LNalg}.

We make two points of comparison. First, in \cite{IriEqui}, Irie uses ECH spectral invariants to establish that $C^\infty$-generically the Reeb orbits of a 3-dimensional Reeb flow are equidistributed. He asks whether in fact the orbit sets realizing the ECH spectral invariants are also asymptotically equidistributed. More precisely, let $(Y,\xi)$ be a contact 3-manifold and let $\Gamma \in H_1(Y;\Z)$ be such that $c_1(\xi)+2\mathrm{PD}(\Gamma)$ is torsion. Consider a sequence of classes $\sigma_k \in ECH(Y,\xi,\Gamma)$ with relative grading $I(\sigma_k,\sigma_0)=2k$. Consider the spectral invariant 
\[c_\sigma(Y,\lambda) = \inf \{L >0:  \sigma \in \im (\iota_L:ECH^{<L}(Y,\lambda,\Gamma) \to ECH(Y,\lambda,\Gamma))\},\]
where $\iota_L$ is the map induced by inclusion of the subcomplex $ECC^{<L}\subset ECC$.
Standard arguments guarantee there is some orbit set $C_k$ with $c_\sigma(Y,\lambda) = \mathcal{A}(C_k)$. Irie asks the following.
\begin{question}[{\cite[\S 1.3]{IriEqui}}]
	For a $C^\infty$-generic contact form $\lambda$, does $C_k/\sqrt{2k}$ weakly converge as a current to $\dd\lambda/\sqrt{\vol(Y,\lambda)}$? I.e. for any one-form $\alpha$ is
\[\lim_{k\to\infty} \frac{1}{\sqrt{2k}} \int_{C_k} \alpha = \frac{1}{\sqrt{\vol(Y,\lambda)}} \int_Y \alpha \wedge \dd\lambda?\]
\end{question}
This question is also related to Hutchings' Ruelle invariant conjecture \cite{HutRuelle}. Of course, assuming they are finite and satisfy a Weyl law, one can pose the same question for the alternative ECH capacities and the orbit sets which represent them.

As a second comparison, Donaldson establishes a kind of equidistribution for his Donaldson divisors \cite[Prop.\!\ 40]{Don}. Namely, if $(X,\w)$ is a closed integral symplectic 4-manifold and if $W_k$ is a Donaldson divisor Poincar\'e dual to $k[\w]$ and constructed as the zero set of an asymptotically holomorphic section of a line bundle, then Donaldson shows that for any two-form $\psi$,
\[\lim_{k\to\infty} \frac{1}{k} \int_{W_k} \psi  = \int_X \psi \wedge \w.\]

We can pose a four-dimensional version of Irie's question, and view Donaldson's result as some positive evidence.

\begin{question}\label{q:alt_equi}
	Let $(X,\w)$ be a closed symplectic 4-manifold with $b_2^+=1$. Consider a generic $J\in \mathcal{J}(X,\w)$ and let $C_k$ be a $J$-holomorphic curve so that $c_k^\Alt(X,\w) = \int_{C_k} \w$. Then is the following property always true (or maybe true if $\w$ is $C^\infty$-generic)?
	\begin{equation}\label{eq:alt_equi}\text{For any $2$-form $\psi$,} \quad \lim_{k\to\infty} \frac{1}{\sqrt{k}}\int_{C_k}\psi = \frac{1}{\sqrt{\vol(X,\w)}} \int_X \psi \wedge \w.
	\end{equation}	
\end{question}

For $\C P^2$, one can take the currents $C_k$ realizing $c_k^\Alt$ to be Donaldson divisors. So Donaldson's result implies \eqref{eq:alt_equi} holds, at least for certain choices of $C_k$,  on $\C P^2$.

\subsection{More on the asymptotics of $e_k^\Alt$}
Once we know that the error term in the Weyl law is bounded, we can occupy ourselves attempting to determine its asymptotic behaviour. Hutchings has conjectured \cite[Conj.\!\ 1.5]{HutRuelle} that for a generic star-shaped domain $(X,\w)$, the subleading asymptotics of the ECH Weyl law satisfy 
\[\lim_{k\to\infty} e_k^\ECH(X,\w) = -\frac{1}{2}\mathrm{Ru}(X).\]
Here $\mathrm{Ru}(X)$ denotes the \emph{Ruelle invariant}, which is a dynamical measure of the average rotation of trajectories of the Reeb flow of $\pd X$. Hutchings proves this conjecture for generic convex and concave toric domains, but very little is known beyond that setting. In addition to offering new connections between ECH and Reeb dynamics, this conjecture has applications to obstructing full fillings (see \cite[Cor.\! 1.13]{HutRuelle}).

Beyond the setting of star-shaped domains, we may wish to understand the subleading asymptotic behaviour of the ECH capacities of a general Liouville domain, or more generally the subleading asymptotics of ``$U$-towers" of ECH spectral invariants as in \cite{Weyl} (c.f.\! Question 1.8 in \cite{EdtPack}). And of course, the elementary alternative capacities and elementary spectral invariants \cite{Hut26} too. The Ruelle invariant is only defined for Reeb flows on integer homology 3-spheres, so it is not clear what to expect outside this setting. We now make a conjecture about the asymptotics of $e_k^\Alt$ for closed 4-manifolds, which we hope can motivate further study of the general subleading asymptotics of ECH-type spectral invariants. 

Define a pseudonorm $\rho : H^2(X;\R)\to\R$ by $\rho(x)=0$ if $\lambda x$ is not in the integral lattice for any $\lambda>0$, and $\rho(x)=1$ if $x$ is a primitive vector in the integral lattice.

\begin{conjecture}\label{conj_ek}
	Let $(X,\w)$ be a closed symplectic 4-manifold with $b_2^+=1$. Then		
	\begin{align*}
			\liminf_{k\to \infty} e_k^\Alt(X,\w) &= -\frac{1}{2}  [\w] \cdot c_1(TX,\w)\\
			 \limsup_{k\to \infty} e_k^\Alt(X,\w) &= -\frac{1}{2}  [\w] \cdot c_1(TX,\w)+ \rho([\w]).
		\end{align*}
In particular, the sequence $e_k^\Alt(X,\w)$ converges if and only if $(X,\w)$ is not rational.
\end{conjecture}
This conjecture can be explicitly verified for some simple cases like symplectic forms on $\C P^2$ and its one- or two-fold blow-up, or on $S^2\times S^2$. Note that \Cref{conj_ek} in conjunction with Hutchings' Ruelle conjecture \cite{HutRuelle} would imply that if $D\subset \R^4$ is a (generic) starshaped domain with smooth boundary which fully fills a closed symplectic 4-manifold $(X,\w)$ with $b_2^+=1$, then
\begin{equation}\label{eq:Ruconj}
	\mathrm{Ru}(\pd D,\lambda_{\std}) \geq  [\w] \cdot c_1(TX,\w).
\end{equation} 
Here, we say $D$ \emph{fully fills} $(X,\w)$ if $\vol(D)=\vol(X,\w)$ and for any $\lambda\in (0,1)$, there exists a symplectic embedding  $(\lambda D,\w_\std) \hookrightarrow (X,\w)$. By Edtmair's recent work \cite{EdtPack}, the ellipsoid 
\[E_a= E\left(\sqrt{\vol(X,\w)a^{-1}},\sqrt{\vol(X,\w)a}\right)\]
fully fills $(X,\w)$ closed for any $a$ sufficiently large. As $a$ grows, $\mathrm{Ru}(\pd E_a,\lambda_\std)$ will go to infinity, so this is consistent with the conjectural inequality \eqref{eq:Ruconj}.

We conclude with the following partial results towards \Cref{conj_ek}.

\begin{proposition}\label{prop:ekconj}
	Let $(X,\w)$ be a closed rational symplectic 4-manifold with $b_2^+=1$. Then
	\begin{align}\
		\liminf_{k\to\infty} e_k^\Alt(X,\w) &\leq -\frac{1}{2}  [\w] \cdot c_1(TX,\w) \label{eq:ek1}\\
		\limsup_{k\to\infty} e_k^\Alt(X,\w) &\leq -\frac{1}{2}  [\w] \cdot c_1(TX,\w)+ 2\frac{\vol(X,\w)}{\rho([\w])}.\label{eq:ek2}
	\end{align}
	If in addition $c_1(TX,\w) = \tau[\w]$ for some $\tau\in\R$, then
	\begin{equation}\label{eq:ek3}
		\liminf_{k\to\infty} e_k^\Alt(X,\w) = -\frac{1}{2}  [\w] \cdot c_1(TX,\w) = -\tau \vol(X,\w).
	\end{equation}
\end{proposition}
Note when $(X,\w)$ is rational and can be rescaled to have volume one, \eqref{eq:ek2} agrees with the claimed value of $\limsup e_k^\Alt$ in \Cref{conj_ek}. Examples where $c_1(TX,\w) = \tau[\w]$ and \eqref{eq:ek3} applies include monotone/Fano symplectic 4-manifolds, like some blow-ups of $\C P^2$, as well as surfaces with torsion first Chern class like the Enriques and hyperelliptic surfaces. Negative monotone examples (i.e. with $\tau<0$) include surfaces of general type with $p_g=0$ and ample canonical bundle, such as the Godeaux surfaces of  Craighero--Gattazzo \cite{DW}.

\begin{proof}[Proof of \Cref{prop:ekconj}]
	The first two inequalities \eqref{eq:ek1} and \eqref{eq:ek2} will follow from our computations in the proof of \Cref{thm:alt_fin}. Rescale $\w$ to be integral and set 
	\[k_n = n^2 \vol +\frac{n}{2} \kappa -\frac{1}{2}b_1,\]
	where $\vol=\vol(x,\w)$ and $\kappa = \langle [\w]\cup c_1(TX,\w),[X]\rangle $. For $n$ sufficiently large, we had that
		\[c_k^\Alt(X,\w)  \leq 2n\vol,\]
		whenever $k_{n-1} < k \leq k_n$. Then, focusing on the sequence of terms $k_n$,
	\begin{align*}
		\liminf_{k\to\infty} e_k^\Alt(X,\w) &\leq \liminf_{n\to\infty} e_{k_n}^\Alt(X,\w)\\
		& \leq \liminf_{n\to \infty} 2n\vol -\sqrt{4k_n\vol}\\
		&= \liminf_{n\to \infty}\frac{4n^2\vol^2 -4k_n\vol}{2n\vol + \sqrt{4k_n\vol}}\\
		&= \liminf_{n\to \infty} \frac{-2n\kappa\vol(X,\w)+ 2b_1\vol}{2n\vol + \sqrt{4n^2\vol^2 +2n\kappa\vol -2b_1\vol}}\\
		&= \liminf_{n\to \infty} \frac{-2\kappa\vol+2n^{-1}b_1\vol}{2\vol + \sqrt{4\vol^2 +2n^{-1}\kappa\vol -2n^{-2}b_1\vol}}\\
		&= -\frac{\kappa}{2}.
	\end{align*}
	This gives \eqref{eq:ek1}. By the same kind of computation, focusing on the terms $k_{n-1}$,
\begin{align*}
		\limsup_{k\to\infty} e_k^\Alt(X,\w) &\leq \limsup_{n\to\infty} \frac{4n^2\vol^2 -4k_{n-1}\vol}{2n\vol + \sqrt{4k_{n-1}\vol}}\\
		&= \limsup_{n\to \infty} \frac{4(2n-1)\vol^2 +2b_1\vol-2(n-1)\kappa\vol}{2n\vol + \sqrt{4(n-1)^2\vol^2 +2(n-1)\kappa\vol -2b_1\vol}}\\
		&= \limsup_{n\to \infty} \frac{4(2-n^{-1})\vol^2 +2b_1n^{-1}\vol-2(1-n^{-1})\kappa\vol}{2\vol + \sqrt{4(1-n^{-1})^2\vol^2 +2(n^{-1}-n^{-2})\kappa\vol -2b_1n^{-2}\vol}}\\
		&= 2\vol(X,\w) -\frac{\kappa}{2}.
	\end{align*}
This holds for any integral symplectic form and we may as well rescale so that $[\w]	$ is primitive. Then, given $\w' = \lambda \w$ for $\lambda>0$, $\rho([\w']) =\lambda$ and
\begin{align*}
	\limsup_{k\to\infty} e_k^\Alt(X,\w') &= \lambda\limsup_{k\to\infty} e_k^\Alt(X,\w)\\
	& \leq 2\lambda \vol(X,\w) - \frac{\lambda}{2} [\w] \cdot c_1(TX,\w)\\
	&= \frac{2\vol(X,\w')}{\rho([\w'])}-\frac{1}{2} [\w'] \cdot c_1(TX,\w').
\end{align*}

To prove \eqref{eq:ek3}, extend $\w$ to an orthogonal basis $\{\mathrm{PD}[\w], e_2, \ldots, e_{b_2}\}$ for $H_2(X;\R)$ so that $e_i^2=-1$ for $2\leq i \leq b_2$. Consider some class $A\in H_2(X;\Z)$ which can be expanded in this basis as
\[A = a_1 [\w] + a_2 e_2+ \cdots + a_{b_2} e_{b_2} \quad \text{for $a_i\in\R$.}\]
Then
\begin{align*}
	A^2&= 2a_1^2 \vol(X,\w) -a_2^2-\cdots - a_{b_2}^2\\
	\w(A) &= 2a_1 \vol(X,\w) \\
	c_1(TX,\w) ( A) &= 2\tau a_1 \vol(X,\w). 
\end{align*}
This gives an upper bound on the Gromov--Taubes index
\[I(A)= A^2 +c_1(TX,\w)\cdot A \leq 2a_1^2 \vol(X,\w) +2\tau a_1 \vol(X,\w).\]
Given a generic collection of $k$ points in $X$ and a generic $\w$-compatible almost-complex structure $J$, standard transversality results imply any somewhere-injective $J$-holomorphic curve $u\in \mathcal{M}^J(X,\w;x_1,\ldots, x_k)$ must have Fredholm index $\ind(u)\geq 2k$. From the adjunction formula, this implies $I([u])\geq 2k$. As in \cite[Rem.\! 2]{HutAlt}, we can restrict to somewhere-injective curves in the definition of the alternative capacities. We conclude that
\begin{equation*}
c_k^\Alt (X,\w) \geq \inf_{\substack{A\in H_2(X) \\ I(A)\geq 2k}} \w(A).	
\end{equation*}
Combined with our above computations, we have that $c_k^\Alt(X,\w)$ is at least the infimum of $2a \vol(X,\w)$ over all $a >0$ which satisfy 
\[2a^2 \vol(X,\w) + 2\tau a\vol(X,\w) \geq 2k. \]
Completing the square, this last inequality implies
\[\left(a+\frac{\tau}{2}\right)^2 \geq \frac{k}{\vol(X,\w)} + \frac{\tau^2}{4},\]
which in turn gives, for $k$ sufficiently large,
\[a \geq \sqrt{\frac{k}{\vol(X,\w)} + \frac{\tau^2}{4}}-\frac{\tau}{2}.\]
Thus for $k$ large,
\[c_k^\Alt(X,\w) \geq \sqrt{4k\vol(X,\w) + \tau^2 \vol(X,\w)^2}-\tau\vol(X,\w).\]
Hence
\begin{align*}
	\liminf_{k\to\infty} e_k^\Alt(X,\w) &\geq \liminf_{k\to\infty}\sqrt{4k\vol(X,\w) + \tau^2 \vol(X,\w)^2}-\tau\vol(X,\w)-\sqrt{4k\vol(X,\w)}\\
	&=\liminf_{k\to\infty} \frac{\tau^2 \vol(X,\w)^2}{\sqrt{4k\vol(X,\w)+\tau^2\vol(X,\w)^2}+ \sqrt{4k\vol(X,\w)}} -\tau\vol(X,\w)\\
	&= -\tau\vol(X,\w).
\end{align*}
Combined with \eqref{eq:ek1}, this gives \eqref{eq:ek3}.
\end{proof}


\begin{thebibliography}{40}
		
	\bibitem[Bei26]{Bei} G. Beiner, Infinite ECH capacities and Anosov flows, arXiv preprint 2606.28316 (2026), 49 pp.
	
	\bibitem[Bir01]{Bir} P. Biran, Lagrangian barriers and symplectic embeddings, Geom. Funct. Anal. {\bf 11} (2001), no.~3, 407--464.
		
	\bibitem[BHO]{BHO} O. Bu\c se, R.~K. Hind and E. Opshtein, Packing stability for symplectic 4-manifolds, Trans. Amer. Math. Soc. {\bf 368} (2016), no.~11, 8209--8222.	
			
	\bibitem[CW21]{CW} J. Chaidez, B. Wormleighton. ECH embedding obstructions for rational surfaces, arXiv preprint 2008.10125	(2021), 23 pp.
		
	\bibitem[Che26a]{Che} G. Chen, Symplectic embeddings of balls into disk prequantization bundles, J. Math. Soc. Japan {\bf 78} (2026), no.~2, 381--408.	
	
	\bibitem[Che26b]{CheU} G. Chen, Computing the U maps on ECH of prequantization bundles, arXiv preprint 2608.16625 (2026), 58 pp.
			
	\bibitem[CCG+]{CCGFHR}  K. Choi, D. Cristofaro-Gardiner, D. Frenkel, M. Hutchings, and   V. G. B. Ramos,  Symplectic embeddings into four-dimensional concave toric domains,  J. Topol. {\bf 7} (2014), no.~4, 1054--1076.	
		
	\bibitem[CM05]{CM} K. Cieliebak and K. Mohnke, Compactness for punctured holomorphic curves, J. Symplectic Geom. {\bf 3} (2005), no.~4, 589--654.	

	\bibitem[CG19]{CG} D. Cristofaro-Gardiner, Symplectic embeddings from concave toric domains into convex ones, J. Differential Geom. {\bf 112} (2019), no.~2, 199--232.
		
	\bibitem[CGH16]{CGH} D. Cristofaro-Gardiner and M. Hutchings, From one Reeb orbit to two, J. Differential Geom. {\bf 102} (2016), no.~1, 25--36.
	
	\bibitem[CGHR]{Weyl} D. Cristofaro-Gardiner, M. Hutchings, and V. G. B. Ramos, The asymptotics of ECH capacities, Invent. Math. {\bf 199} (2015), no.~1, 187--214.
			
	\bibitem[DHL]{DHL} B. Dai, C.-I. Ho and T.-J. Li, Nonorientable Lagrangian surfaces in rational 4-manifolds, Algebr. Geom. Topol. {\bf 19} (2019), no.~6, 2837--2854.
		
	\bibitem[DW99]{DW} I.~V. Dolgachev and C. Werner, A simply connected numerical Godeaux surface with ample canonical class, J. Algebraic Geom. {\bf 8} (1999), no.~4, 737--764.	
		
	\bibitem[Don96]{Don} S. Donaldson, Symplectic submanifolds and almost complex geometry, J. Differential Geom. {\bf 44} (1996), no.~4, 666--705.
		
	\bibitem[Edt25a]{EdtPFH} O. Edtmair, An elementary alternative to PFH spectral invariants, J. Symplectic Geom. {\bf 23} (2025), no.~3, 511--574.	
	
	\bibitem[Edt25b]{EdtPack} O. Edtmair, Packing stability and the subleading asymptotics of symplectic Weyl laws, arXiv preprint 2509.15390 (2025), 62 pp.
		
	\bibitem[EGH]{EGH} Y. Eliashberg, A. Givental and H. Hofer, Introduction to symplectic field theory, Geom. Funct. Anal. {\bf 2000}, Special Volume, Part II, 560--673.
		
			
	\bibitem[ET98]{ET} Y. Eliashberg and W. Thurston, \emph{Confoliations}, Vol. 13. University Lecture Series. American Mathematical Society, 1998.
		
		\bibitem[FR22]{FR} B. Ferreira and V. G. B.  Ramos, Symplectic embeddings into disk cotangent bundles, J. Fixed Point Theory Appl. {\bf 24} (2022), no.~3, Paper No. 62, 31 pp.
	
	\bibitem[FRV]{FRV} B. Ferreira, V.~G.~B. Ramos and A. Vicente, Gromov width of the disk cotangent bundle of spheres of revolution, Adv. Math. {\bf 487} (2026), Paper No. 110761, 41 pp.
	
	\bibitem[FH22]{FH} J. Fish and H. Hofer, Almost existence from the feral perspective and some questions, Ergodic Theory Dynam. Systems {\bf 42} (2022), no.~2, 792--834.	
		
	\bibitem[GP20]{GP} S. Ganatra and D.~M. Pomerleano, Symplectic cohomology rings of affine varieties in the topological limit, Geom. Funct. Anal. {\bf 30} (2020), no.~2, 334--456.	
		
	\bibitem[GS24]{GS} S. Ganatra and K. Siegel, On the embedding complexity of Liouville manifolds, J. Differential Geom. {\bf 127} (2024), no.~3, 1019--1082.	
		
	\bibitem[GS09]{GaS} D.~T. Gay and A.~I. Stipsicz, Symplectic surgeries and normal surface singularities, Algebr. Geom. Topol. {\bf 9} (2009), no.~4, 2203--2223.	
		
	\bibitem[Gir17]{Gir} E. Giroux, Remarks on Donaldson's symplectic submanifolds, Pure Appl. Math. Q. {\bf 13} (2017), no.~3, 369--388.
	
	\bibitem[HZ94]{HZ} H. Hofer and E. Zehnder, {\it Symplectic invariants and Hamiltonian dynamics}, reprint of the 1994 edition,  Modern Birkh\"auser Classics, Birkh\"auser Verlag, Basel, 2011.
	
	\bibitem[Hut02]{Hutindineq} M. Hutchings, An index inequality for embedded pseudoholomorphic curves in symplectizations, J. Eur. Math. Soc. (JEMS) {\bf 4} (2002), no.~4, 313--361.
	
	\bibitem[Hut09]{Hutind} M. Hutchings, The embedded contact homology index revisited, in {\it New perspectives and challenges in symplectic field theory}, 263--297, CRM Proc. Lecture Notes, 49, Amer. Math. Soc., Providence, RI.
	
	\bibitem[Hut11]{HutQ} M. Hutchings, Quantitative embedded contact homology, J. Differential Geom. {\bf 88} (2011), no.~2, 231--266.
	
	\bibitem[Hut14]{Hutnotes} M. Hutchings, Lecture notes on embedded contact homology, in
	\emph{Contact and Symplectic Topology}, 389--484, Bolyai Soc. Math. Stud. 26 (2014), J\'anos Bolyai Mathematical Society, Budapest, Springer.
	
	\bibitem[Hut22a]{HutAlt} M. Hutchings, An elementary alternative to ECH capacities, Proc. Natl. Acad. Sci. U.S.A. \textbf{119} (2022), no. 35, 10 pp..
	
	\bibitem[Hut22b]{HutRuelle} M. Hutchings, ECH capacities and the Ruelle invariant, J. Fixed Point Theory Appl. {\bf 24} (2022), no.~2, Paper No. 50, 25 pp.
	
	\bibitem[Hut24]{Hutqc} M. Hutchings, Elementary spectral invariants and quantitative closing lemmas for contact three-manifolds, J. Mod. Dyn. {\bf 20} (2024), 635--662.
	
	\bibitem[Hut26]{Hut26} M. Hutchings, Elementary spectral invariants and three-dimensional Reeb dynamics, arXiv preprint 2605.12958 (2026), 51 pp.
	
	\bibitem[Hut]{HutTQFT} M. Hutchings, Embedded contact homology as a (symplectic) field theory, in preparation.
	
	\bibitem[Iri15]{Iri} K. Irie, Dense existence of periodic Reeb orbits and ECH spectral invariants, J. Mod. Dyn. {\bf 9} (2015), 357--363.
	
	\bibitem[Iri21]{IriEqui} K. Irie, Equidistributed periodic orbits of $C^\infty$-generic three-dimensional Reeb flows, J. Symplectic Geom. {\bf 19} (2021), no.~3, 531--566.
	
	\bibitem[Lee04]{Lee} J. Lee, Family Gromov--Witten invariants for K\"ahler surfaces, Duke Math. J. {\bf 123} (2004), no.~1, 209--233.
		
	\bibitem[LP07]{LP} J. Lee and T.~H. Parker, A structure theorem for the Gromov--Witten invariants of K\"ahler surfaces, J. Differential Geom. {\bf 77} (2007), no.~3, 483--513.
	
	\bibitem[LL99]{LL} T.-J. Li and A.-K. Liu, The equivalence between ${\rm SW}$ and ${\rm Gr}$ in the case where $b^+=1$, Internat. Math. Res. Notices {\bf 1999}, no.~7, 335--345.
	
	\bibitem[LM19]{LM} T.-J. Li and C.~Y. Mak, Symplectic divisorial capping in dimension 4, J. Symplectic Geom. {\bf 17} (2019), no.~6, 1835--1852.
	
	\bibitem[LN24]{LNalg} T.-J. Li and S. Ning, Algebraic capacities as tropical polynomials over the reduced $c_1$-positive symplectic cone, Trans. Amer. Math. Soc. {\bf 377} (2024), no.~12, 8381--8409.
	
	\bibitem[LN26]{LN} T.-J. Li and S. Ning, Symplectic log Kodaira dimension $-\infty$, affine-ruledness, and unicuspidal rational curves, arXiv preprint 501.14668 (2026), 53 pp..
	
	\bibitem[MNR+]{ECHplumb} A. Marinkovi\'c, J. Nelson, A. Rechtman, L. Starkson, S. Tanny, and L. Wang, Properties of contact toric structures and concave boundaries of linear plumbings, arXiv preprint 2501.08451 (2026), to appear in Trans. Am. Math. Soc., 68 pp.
	
	\bibitem[MT25]{MT} T. Mark and B. Tosun, On Weinstein domains in symplectic manifolds, arXiv preprint 2512.04278 (2025), 27 pp.
				
	\bibitem[McD95]{McD} D. McDuff, Lectures on Gromov invariants for symplectic $4$-manifolds, in {\it Gauge theory and symplectic geometry (Montreal, PQ, 1995)}, 175--210, NATO Adv. Sci. Inst. Ser. C: Math. Phys. Sci., 488, Kluwer Acad. Publ., Dordrecht.			
				
	\bibitem[MS96]{MS96} D. McDuff and D. Salamon, A survey of symplectic $4$-manifolds with $b^{+}=1$, Turkish J. Math. {\bf 20} (1996), no.~1, 47--60.		
	
	\bibitem[McL16]{McL} M. McLean, Reeb orbits and the minimal discrepancy of an isolated singularity, Invent. Math. {\bf 204} (2016), no.~2, 505--594.
	
	\bibitem[MR25]{MR} M. Miranda and V. G. B. Ramos, Embedded contact homology of the unit cotangent bundle of the Klein bottle, arXiv preprint 2508.06400 (2025), 70 pp.
	
	\bibitem[Ops13]{Ops} E. Opshtein, Singular polarizations and ellipsoid packings, Int. Math. Res. Not. IMRN {\bf 2013}, no.~11, 2568--2600.
	
	\bibitem[Pra25]{Pra} R. Prasad, High-dimensional families of holomorphic curves and three-dimensional energy surfaces, arXiv preprint 2405.01106 (2025), to appear in Ann. Math., 61 pp.
	
	\bibitem[Tau98]{Tau} C.~H. Taubes, The structure of pseudo-holomorphic subvarieties for a degenerate almost-complex structure and symplectic form on $S^1\times B^3$, Geom. Topol. {\bf 2} (1998), 221--332.
	
	\bibitem[Tau00]{TauSWGr} C.~H. Taubes, {\it Seiberg Witten and Gromov invariants for symplectic $4$-manifolds}, First International Press Lecture Series, 2, International Press, Somerville, MA, 2000.
	
	\bibitem[Ush12]{Ush} M. Usher, Many closed symplectic manifolds have infinite Hofer--Zehnder capacity, Trans. Amer. Math. Soc. {\bf 364} (2012), no.~11, 5913--5943.

	\bibitem[Wen15]{Wendl} C. Wendl, \emph{Lectures on holomorphic curves in symplectic and contact geometry}, Version 3.3 (2015), retrieved from \url{https://www.mathematik.hu-berlin.de/~wendl/publications.html#notes}. 

	\bibitem[Zeh86]{Zeh} E. Zehnder, Remarks on periodic solutions on hypersurfaces, in {\it Periodic solutions of Hamiltonian systems and related topics (Il Ciocco, 1986)}, 267--279, NATO Adv. Sci. Inst. Ser. C: Math. Phys. Sci., 209, Reidel, Dordrecht.

\end{thebibliography}
\end{document}